\documentclass{amsart}

\usepackage[letterpaper,margin=1.2in]{geometry}

\usepackage{amsmath,amssymb,amsthm}
\usepackage{MnSymbol}
\usepackage{enumerate}
\usepackage{hyperref}
\usepackage{mathtools}
\usepackage{bbm}
\usepackage{xcolor}
\usepackage{tikz}
\usepackage{tikz-cd}

\usepackage{comment}

\newtheorem{theorem}{Theorem}[section]
\newtheorem{lemma}[theorem]{Lemma}
\newtheorem{proposition}[theorem]{Proposition}
\newtheorem{corollary}[theorem]{Corollary}

\newtheorem*{theorem*}{Theorem}

\numberwithin{equation}{section}
\numberwithin{figure}{section}

\theoremstyle{definition}
\newtheorem{observation}[theorem]{Observation}
\newtheorem{definition}[theorem]{Definition}
\newtheorem{notation}[theorem]{Notation}
\newtheorem{assumption}[theorem]{Assumption}

\theoremstyle{remark}
\newtheorem{remark}[theorem]{Remark}
\newtheorem{example}[theorem]{Example}

\newcommand\R{\mathbb{R}}
\newcommand\C{\mathbb{C}}

\newcommand\N{\mathbb{N}}

\DeclareMathOperator{\id}{id}

\DeclareMathOperator{\re}{Re}

\DeclareMathOperator{\Tr}{Tr}
\DeclareMathOperator{\tr}{tr}

\DeclareMathOperator{\sa}{sa}

\DeclareMathOperator{\Lip}{Lip}
\DeclareMathOperator{\NC}{NC}

\DeclarePairedDelimiter{\norm}{\lVert}{\rVert}
\DeclarePairedDelimiter{\ip}{\langle}{\rangle}

\begin{document}
	
	\title{On counterexamples to the unification of free entropy}
	\author{David Jekel}
	
	\address{Department of Mathematics and Statistics, University of Ottawa, 150 Louis-Pasteur Pvt, Ottawa, ON, K1N 6N5, Canada}
	\email{djekel@uottawa.ca}
	
	\maketitle
	
	\begin{abstract}
		It has been a long-standing open problem to determine when two versions of Voiculescu's free entropy, $\chi$ and $\chi^*$, agree.  OpenAI recently generated the first example with $-\infty < \chi < \chi^*$.  We give a detailed presentation of a similar example, and explicitly show the separation between free Fisher information and the large-$n$ limit of classical Fisher information through the relationship with Wasserstein distance.  We thus situate the example in the context of probability theory and information geometry.
	\end{abstract}
	
	\section{Introduction}
	
	\subsection{Historical motivation}
	
	Voiculescu's work on free probability provided a wonderful link between operator algebras and random matrix theory \cite{Voiculescu1991}.  A von Neumann algebra $M$ with an appropriate tracial state $\tau$ can be viewed as a noncommutative analog of $L^\infty(\Omega,\mathbb{P})$ for a probability space $(\Omega,\mathbb{P})$ with $\mathbb{E}$ as the analog of the expectation.  The behavior of free products of von Neumann algebras is described by free independence, just as tensor products are described by classical independence.  Voiculescu showed that many random $n \times n$ matrix models become freely independent in the large-$n$ limit, which enabled surprising applications to operator algebras (e.g. \cite{VoiculescuFE3}, \cite{Hayes2022PT}).
	
	Free probability has developed many analogs of concepts from classical probability.  In particular, a series of papers investigated analogs of the differential entropy $h(\mu)$ of a probability measure.  While free entropy can be defined explicitly in the case of a single self-adjoint variable using logarithmic energy \cite{VoiculescuFE1}, the multivariate setting is much more challenging since there is not a straighforward analog of the probability density of some tuple $(X_1,\dots,X_m)$ of noncommutative random variables.  The \emph{microstates approach} of \cite{VoiculescuFE2} follows Boltzmann's mantra that the entropy of a macrostate is the log of the number of microstates; $\chi$ is thus defined as the $\limsup$ of normalized log-volumes of spaces of matrix approximations.  Another convenient characterization is that $\chi$ is the largest possible limit of normalized classical entropies of random matrix models for the given law \cite[Appendix B]{ST2022fractional}.
	
	Meanwhile, another version of free entropy $\chi^*$ is defined via the relationship between entropy and Fisher information \cite{VoiculescuFE5}.  It is well known in classical probability that if $X$ is a random variable in $\R^m$ and $Z$ is an independent standard Gaussian, then $(d/dt) h(X + t^{1/2}Z)$ is given by half the $L^2$-norm of the score function (i.e.\ the Fisher information).  In the free setting, score functions (also known as \emph{conjugate variables}) cannot be expressed directly since there is no density, but they can be defined via an integration-by-parts relation.  Voiculescu thus defined $\chi^*$ using the integral of the free Fisher information of $X + t^{1/2} Z$ where $Z$ is a freely independent semicircular family.
	
	It remained an open problem for a long time whether $\chi = \chi^*$ in general.  The first major progress was the inequality $\chi \leq \chi^*$ due to Biane, Capitaine, and Guionnet \cite{BCG2003}.  This follows from a large-deviations upper bound for free Brownian motion.  The quantity $\chi$ fits naturally with large deviations theory, as it is defined as the limiting exponential growth rate of probabilities of events involving the noncommutative moments.  Thus, a complete large deviations theory for the Gaussian unitary ensemble amounts to showing that the $\limsup$ and $\liminf$ versions of $\chi$ agree and describing them using some analytically tractable rate function (for instance, $\chi^*$).  Biane, Capitaine, and Guionnet used an important technique in large deviations theory, namely evaluating the rate function as the solution of a stochastic optimization problem, an insight that was powerfully expressed for classical Brownian motion in the formula of Bou{\'e} and Dupuis \cite{BoueDupuis1998variational} around the same time.
	
	The relationship between $\chi$ and $\chi^*$ is well understood in the setting of free Gibbs laws  \cite{GMS2006,GS2009,JLS2022,GMS2026}, which are analogs of measures with density $e^{-V}$, provided that the potential is convex.  These convex free Gibbs laws are thus analogs of log-concave measures; this gives us access to powerful tools such as Langevin diffusion, the log-Sobolev and Talagrand inequalities, and concentration of measure (see \cite{OttoVillani2000} \cite{AGZ2009}).  Dabrowski \cite{Dabrowski2017Laplace} showed that $\chi(X) = \chi^*(X)$ when $X$ is an $m$-tuple associated to a convex potential $V$ with suitable hypotheses; I gave another similar result in \cite{JekelEntropy}.  Beyond the convex setting, the only potentials where significant progress has been made are the \emph{trapping} potentials of \cite{GMS2026}; an explicit asymptotic expansion of the free entropy in powers of $n$ was given by Parraud and Schnelli \cite{ParraudSchnelli2025}, although the problem of showing $\chi = \chi^*$ in this case has not been explicitly addressed.
	
	The first negative result on $\chi$ versus $\chi^*$ is a corollary of the resolution of the Connes embedding problem in \cite{JNVWY2020}.  Namely, there exist noncommutative random variables $(X_1,\dots,X_m)$ which are not \emph{Connes-embeddable}, i.e.\ they do not admit \emph{any} matrix approximations, and therefore have $\chi(X) = -\infty$ trivially.  They can also be arranged to satisfy $\chi^*(X) > -\infty$ using a free semicircular perturbation.  While strictly speaking this does not say anything about whether $\chi = \chi^*$ for all Connes-embeddable noncommutative random variables, the complexity-theoretic content does perhaps give a hint that free entropy could be very badly behaved.  Indeed, \cite{JNVWY2020} shows that various optimization problems involving families of projections cannot be solved algorithmically.  From this, it is plausible that stochastic optimization problems related to large deviations theory might also behave wildly in the large-$n$ limit.
	
	Further hints come from the study of optimal transport in the free setting.  Entropy resides naturally in the habitat of Wasserstein geometry, which looks at the space of probability measures on $\R^m$ (or more general spaces) equipped with the $L^2$-Wasserstein distance coming from optimal transport theory.  The Wasserstein distance can also be viewed as the geodesic distance associated to a kind of Riemannian structure on $\mathbb{P}(\R^m)$, while various functional inequalities relating entropy and Wasserstein distance can be viewed as statements about convexity or concavity of entropy near a point in the Wasserstein space.  Jordan, Kinderleher, and Otto \cite{JKO1998,Otto2001} showed that the gradient flow of entropy on the Wasserstein space is given by the heat evolution of the measure; here the heat evolution $\mu_t$ of an initial measure $\mu$ refers the distribution of $X + t^{1/2} Z$ where $X \sim \mu$ and $Z$ is an independent standard Gaussian.  The score function is exactly the Wasserstein gradient of entropy, which is why the Fisher information describes the rate of change of entropy along this flow.  For more background on Wasserstein gradient flows, see \cite{AGS2008gradient,BGL2014}.

	In the free probability setting, Biane and Voiculescu defined an analog of Wasserstein distance \cite{BV2001}.  Subsequently free analogs of the Gaussian Talagrand inequality were proved in \cite{HPU2004,HU2006}.  In the study of free Gibbs laws, analogs of optimal transport and diffusion constructions were used to prove free transport results \cite{GS2014,JekelExpectation,DGS2021,JLS2022} which established a smooth change of variables between the original tuple and a free semicircular family (and in particular isomorphism of the associated von Neumann and $\mathrm{C}^*$-algebras).
	
	The joint work \cite{GJNS2021} undertook a more systematic study of optimal transport in the noncommutative setting, including an analog of Monge-Kantorovich duality.  At the same time, it showed several new subtleties in the noncommutative setting.  In particular, there exist matrix tuples for which an optimal coupling in the sense of Biane and Voiculescu \cite{BV2001} can only exist in a non-Connes-embeddable von Neumann algebra.  Hence, it may be impossible to realize the abstract free Wasserstein distance by any matrix approximations at all.
	
	This was one of the motivations for my approach in \cite{JekelModelEntropy} of integrating continuous model theory with free probability.  It seemed that the noncommutative law giving the moments of the tuple $(X_1,\dots,X_m)$ did not provide enough information about $X$ for entropy and Wasserstein distance to have the correct relationship.  Indeed, the value of optimization problems such as $\inf_Y \re \tr(p(X,Y))$ is not uniquely determined by the law of $X$, which in model-theoretic terms is exactly the fact that tracial von Neumann algebras usually do not admit quantifier elimination \cite{Farah2024,FJP2026}.  However, we could try to replace the noncommutative law of $X$ by the full \emph{type} of $X$ which captures the value of optimization problems with iterated suprema and infima (which are the quantifiers in continuous model theory).  This led to a Monge--Kantorovich duality for types in \cite{JekelTypeCoupling}, an analysis relating entropy and Wasserstein distance for types in \cite{Jekel2025InfoGeom}, and the Wasserstein gradient flow of entropy in the metric sense in \cite{Jekel2026FIG}.
	
	In \cite[\S 4.4]{Jekel2025InfoGeom}, I gave an example showing another obstruction to na{\"i}ve conjectures about free Wasserstein distance and entropy.  Namely, there are some noncommutative variables $X = (X_1,X_2,X_3)$ and $Y = (Y_1,Y_2,Y_3)$ such that there do not exist any random matrix models $X^{(n)}$ and $Y^{(n)}$ satisfying \emph{both} that classical entropy of $X^{(n)}$ converges to $\chi(X)$ \emph{and} that the classical Wasserstein distance converges to the free Wasserstein distance restricted to Connes-embeddable couplings.  Hence, using only the noncommutative law rather than the type, it is not possible to study free information geometry directly as the limit of classical information geometry on the matrix space.  This gives an even stronger hint that the Wasserstein gradient flow of the microstate entropy $\chi$ might not be adequately described using only the noncommutative law (or in particular, in terms of $\chi^*$).
	
	OpenAI recently posted the first known counterexample where $-\infty < \chi < \chi^*$ \cite{OpenAI2026entropy}, conclusively dashing the original hope that $\chi$ might agree with $\chi^*$ in a sufficiently general setting and give a simple description of the large deviations problem for random matrices.

	\subsection{Main results}
	
	This paper discusses a variant of the example given by \cite{OpenAI2026entropy}, with the following setup.  Note our convention of using $e^{-t/2} X + (1 - e^{-t})^{1/2}S$ rather than $X + t^{1/2}S$ is based on the gradient flow for the entropy relative to the Gaussian measure, rather than the Lebesgue measure.  These are related by a simple change of variables, but the Gaussian convention is convenient for some purposes because the variance does not grow as $t \to \infty$.
	
	\begin{assumption} \label{ass: main example setup}
		Fix $m \geq 2$.  Let $X = (X_1,\dots,X_{m+1})$ where $(X_1,\dots,X_m)$ is a standard free semicircular family, $X_{m+1}$ is a standard semicircular operator, and $(X_1,\dots,X_m)$ and $X_{m+1}$ are tensor independent.  Let $S = (S_1,\dots,S_{m+1})$ be a standard semicircular family freely independent from $X$, and set
		\begin{equation} \label{eq: main construction of X t}
			X(t) = e^{-t/2} X + (1 - e^{-t})^{1/2} S.
		\end{equation}
	\end{assumption}
	
	\begin{remark}
	The proof of OpenAI takes $X_{m+1}$, instead of a semicircular, to be an operator with discrete spectrum evenly distributed on equally spaced points in $[-1,1]$.  The spectral projections of $X_{m+1}$ are then used to exhibit a block structure for $X_1$, \dots, $X_m$ since $X_{m+1}$ commutes with them.  However, it is reasonable to replace $X_{m+1}$ with an operator with diffuse spectral measure, such as a semicircular operator, since from it one can form any desired number of spectral projections to produce a block structure.
	\end{remark}
	
	\begin{remark}
	The example above also has some similarity with \cite[\S 4.4]{Jekel2025InfoGeom}, which looked a $3$-tuple $(X_1,X_2,X_3)$ where $X_1$ and $X_2$ commute with $X_3$, and $X_3$ has variance $\varepsilon$.  In both cases, the last variable contributes very little to the variance or the Wasserstein distance compared to the size of the overall tuple; in the instance of \cite{Jekel2025InfoGeom}, this is arranged by making the variance of the last variable $\varepsilon$, and in the case of \cite{OpenAI2026entropy}, it is arranged by making the number of variables $X_1$, \dots, $X_m$ large, so that the total variance $\sqrt{m}$ is much larger than the contribution of $X_{m+1}$.  Despite having a small contribution to the variance, the last variable still has a very large effect on the behavior of the tuple overall; indeed, the commutation and resulting block structure affect all the previous coordinates $X_1$, \dots, $X_m$ in a way that is not detectable from the distribution of $(X_1,\dots,X_m)$ alone, which is simply free semicircular.
	\end{remark}
	
	We will then prove the following variant of \cite[Theorem 1.1]{OpenAI2026entropy}.
	
	\begin{theorem} \label{thm: main entropy}
		Under Assumption \ref{ass: main example setup}, let $m$ be sufficiently large\footnote{According to Wolfram Alpha, this threshold for $m$ is reached when $m \approx 5.4694 \times 10^{15}$.} that
		\begin{equation} \label{eq: main condition on m}
			\frac{1}{4} (m+1)^{1/8} \geq 17 + \sqrt{2 + \log(m+1)},
		\end{equation}
		and let
		\begin{equation} \label{eq: main condition on t}
			0 < t < \frac{1}{(m+1)^3}.
		\end{equation}
		Then
		\begin{equation} \label{eq: main conclusion}
		-\infty < \chi(X(t)) < \chi^*(X(t)).
		\end{equation}
	\end{theorem}
	
	One of the facts from \cite{JekelPi2024} used in \cite{OpenAI2026entropy} is that the difference $\chi(X(t)) - \chi(X(t'))$ can be estimated using the large-$n$ $\liminf$ of classical Fisher information of random matrix models $X^{(n)}(t)$ for $X(t)$ obtained as $e^{-t/2} X^{(n)}(0) + (1 - e^{-t})^{1/2} S^{(n)}$ for a GUE matrix tuple $S^{(n)}$.
	This $\liminf$ of the classical Fisher information gives an upper bound for the free Fisher information $X(t)$.  Hence, separation of $\chi$ and $\chi^*$ must correspond to a separation between the free Fisher information $\Phi^*(X(t))$ and the $\liminf$ of the normalized classical Fisher information $\mathcal{I}^{(n)}(X^{(n)}(t))$.  However, the upper bound for $\chi$ was proved by OpenAI using volume estimation and the data processing inequality, and did not directly give a bound on the classical Fisher information of the random matrix models.  We aim here to exhibit the separation of Fisher information explicitly.
	
	To state this result carefully, we need start the process at some small positive time $t_0$ rather than at time zero.  The reason is that $X = X(0)$ has entropy $-\infty$ due to nontrivial commutation, and hence would not be a valid counterexample.  So the initial input needs to be a matrix model for $X(t_0)$ rather than $X(0)$.  We therefore set up the following notation and objects.
	
	\begin{assumption} \label{ass: setup for Y}
		Furthermore, fix $t_0 > 0$ and another free semicircular family $Z$ freely independent from $X$ and $S$, and let
		\begin{equation}
		Y(t) = e^{-t/2} X(t_0) + (1 - e^{-t})^{1/2} Z.
		\end{equation}
		Note that since $e^{-t/2} (1 - e^{-t_0})^{1/2} S + (1 - e^{-t})^{1/2} \tilde{S}$ is free from $X$ and has the same distribution as $(1 - e^{-t_0-t})^{1/2} S$, we have
		\begin{equation}
		Y(t) \sim X(t_0 + t).
		\end{equation}
	\end{assumption}
	
	\begin{assumption} \label{ass: matrix models for Y}
		Let $Y^{(n)}(0)$ be a sequence of random matrix $m+1$-tuples such that
		\begin{enumerate}[(1)]
			\item For all noncommutative polynomials $f$, we have $\tr_n[f(Y^{(n)}(0))] \to \tau[f(Y(0))]$ in probability.
			\item We have $\norm{Y^{(n)}(0)}_\infty \leq 3$.
		\end{enumerate}
		Let $Z^{(n)}$ be an $(m+1)$-tuple with the multivariate GUE distribution, independent from $Y^{(n)}(0)$, and set
		\begin{equation}
		Y^{(n)}(t) = e^{-t/2} Y^{(n)}(0) + (1 - e^{-t})^{1/2} Z^{(n)}.
		\end{equation}
	\end{assumption}
	
	Then our main result comparing classical and free Fisher information is the following.
	
	\begin{theorem} \label{thm: main Fisher}
		Under Assumptions \ref{ass: main example setup}, \ref{ass: setup for Y}, \ref{ass: matrix models for Y}, suppose that \eqref{eq: main condition on m} holds, and let
		\begin{equation}
			0 < t_0 \leq (m+1)^{-3}, \qquad t = 3 \log 2 + \log(1+m^{1/2}) + t' \text{ where } 0 \leq t' \leq \frac{1}{4} \log (m+1).
		\end{equation}
		Then
		\begin{equation} \label{eq: main Fisher}
		\liminf_{n \to \infty} \mathcal{I}^{(n)}(Y^{(n)}(t)) > \Phi^*(Y(t))
		\end{equation}
		where $\mathcal{I}^{(n)}$ is the normalized classical Fisher information (see Notation \ref{not: normalized entropy and Fisher}).
	\end{theorem}
	
	\subsection{Proof techniques}
	
	The proof has two parts:
	\begin{enumerate}[(1)]
		\item A lower bound on $\chi^*$ corresponding to upper bounds on the free Fisher information $\Phi^*$.
		\item An upper bound on $\chi$ correspond to lower bounds on the limiting classical Fisher information.
	\end{enumerate}
	
	For item (1), the proof of OpenAI is based on combinatorial free probability.  It gives an upper bound for the Fisher information of any tuple $X$ using a power series involving free cumulants of $X$ \cite[Proposition 3.2]{OpenAI2026entropy} (see Proposition \ref{prop: main conjugate variable bound}).  This is a consequence of a formula expressing $X$ as a power series in the higher-order conjugate variables of Voiculescu \cite{VoiculescuFE5} (see Proposition \ref{prop: conjugate variable identity}) based on an infinite-order Stein identity or integration-by-parts relation (see Proposition \ref{prop: Stein identity}); the classical analog of this result is \cite[Corollary 1]{Barbour1986} \cite[Proposition 2.2]{NP2010cumulants} in the single-variable case, and the same idea is used in the multivariate setting in \cite[\S 3-4]{BKL2026KLS} to address the KLS conjecture in convex geometry.  These computations are closely related to the theory of higher-order Stein kernels due to Fathi \cite{Fathi2019Stein,Fathi2021higher} (see especially \cite[Lemma 3.2]{Fathi2021higher}), which was adapted to the free setting by Diez \cite[\S 8]{Diez2023}, as both computations boil down to Taylor expansion and integration by parts.  For previous work on free Stein kernels and related ideas, see \cite{FathiNelson2017,Diez2023,CFM2020,Cebron2021,CN2021,CN2025}.
	
	Just as in \cite[\S 8.1, Lemma 4]{Diez2023} and \cite[Lemma 3.2]{Fathi2021higher}, the outcome of the computation based on \cite{OpenAI2026entropy} is that if the moments up to order $k$ of $X$ agree with those of the semicircular variable, then the Fisher information should be bounded by a constant times $e^{-(k+1)t}$ for large $t$.  OpenAI performed the computation using an infinite series rather than the finite truncated series used in the Stein kernel setting.  The infinite-series approach requires a lot more work to implement in the classical setting since it requires estimating the higher-order cumulants, while in the free case the cumulants grow at most geometrically in the degree, so power series methods are even more natural.
	
	The resulting estimate (\cite[Proposition 3.2]{OpenAI2026entropy}, Proposition \ref{prop: main conjugate variable bound}) implies that if $X(t) = e^{-t/2}X + (1 - e^{-t})^{1/2} S$ for a freely independent semicircular family, and if the moments of order up to $k$ agree with those of the semcircular family, then
	\begin{equation}
		\Phi_g^*(X(t)) \lesssim e^{-(k+1)t} \text{ for large } t,
	\end{equation}
	where $\Phi_g^*$ is the ``Gaussian'' version of Fisher information (see Definition \ref{def: Gaussian free Fisher}).
	After giving a general estimate for free Fisher information using free cumulants, OpenAI estimated the free cumulants of its example model \cite[Lemma 4.1]{OpenAI2026entropy}; in particular showing that many of the cumulants were zero due to the free independence of $X_1$, \dots, $X_m$ and tensor independence with $X_{m+1}$.  Whereas its argument was based on M{\"o}bius inversion, we can apply existing formulas for free cumulants in the setting of graph independence, which allows mixtures of free and tensor independence \cite{Mlotkowski2004,SpWys2016}.  This is especially convenient as it allows explicit computation for our choice of $X_{m+1}$ being a standard semicircular; see Proposition \ref{prop: graph independent bound}.  This results in very good bounds on the free Fisher information for large enough $t$, namely, letting $\Phi_g^*$ be the free Fisher information relative to the Gaussian, or $\Phi_g^*(Y) = \Phi^*(Y) + \norm{Y}_2^2 - 2(m+1)$, we have
	\begin{equation}
	\Phi_g^*(Y(t)) \lesssim (m+1)^2 e^{-4t} \text{ when } e^{-t} \lesssim (m+1)^{-1/2};
	\end{equation}
	see \eqref{eq: example free Fisher upper bound 2} for more precise statement.
	
	For item (2), OpenAI argued that $Y^{(n)}(t)$ must have a ``persistent block structure'' that prevents it from having large enough classical entropy.  Indeed, $Y_{m+1}(0)$ approximately commutes with $Y(0)$ when $t_0$ is small (recall $Y(0)$ corresponds to $X(t_0)$ in our notation), and hence for large $n$, $Y_{m+1}^{(n)}(0)$ must also approximately commute with $Y^{(n)}(0)$ with probability tending to $1$.  Hence, spectral projections of $Y_{m+1}^{(n)}(0)$ can be used to give an approximate block decomposition of $Y^{(n)}(0)$ that still sees most of the mass of $Y^{(n)}(0)$.  Namely, there exist projections $(p_j)_{j=1}^k$ of trace $\approx 1/k$ such that $\sum_{j=1}^k p_j = 1$ and
	\begin{equation}
		\sum_{j=1}^k p_j Y^{(n)}(0) p_j \approx Y^{(n)}(0).
	\end{equation}
	On the other hand, the GUE matrix $Z^{(n)}$ cannot have any such block structure, or
	\begin{equation}
		\sum_{j=1}^k p_j Z^{(n)} p_j \approx 0.
	\end{equation}
	This means that
	\begin{equation}
		\sum_{j=1}^k p_j Y^{(n)}(t) p_j \approx e^{-t/2} Y^{(n)}(0).
	\end{equation}
	This allows one to show that the distribution of $Y^{(n)}(t)$ cannot be too close to the distribution of $Z^{(n)}$ since the variance of $\sum_j p_j Y^{(n)}(t) p_j$ is about $(m+1)e^{-t}$ while the variance of $\sum_j p_j Z^{(n)} p_j$ is close to zero.
	
	To make these approximate statements rigorous, we must first of all account for the fact that the projections $p_j$ vary with $Y^{(n)}(0)$, and they cannot a priori be expressed as a function of $Y^{(n)}(t)$.  Thus, we consider the following function
	\begin{equation}
	\psi_k^{(n)}(x) = \sup_{p} \norm*{ \sum_{j=1}^k p_jxp_j }_2,
	\end{equation}
	for $x \in \mathbb{M}_n)_{\sa}^m$, where $\norm{\cdot}_2$ is the normalized Hilbert--Schmidt norm of a matrix tuple and $P$ ranges of choices of projections with trace about $1/k$ (see Notation \ref{not: partition test} for the precise definition).
	
	OpenAI considers the Borel event where $\psi_k^{(n)}$ is larger than a certain threshold, and then compares its probability under the Gaussian distribution versus the distribution of $Y^{(n)}(t)$.  To show that the value of $\psi_k^{(n)}(S^{(n)})$ is small with high probability, it uses standard concentration of measure arguments \cite[Lemma 6.1]{OpenAI2026entropy}.  Then it uses the data processing inequality describing the behavior of relative entropy (i.e.\ KL divergence) under conditioning onto a finite partition of the space \cite[p. 13, proof of Theorem 1.1]{OpenAI2026entropy}.  We remark that the idea of estimating the volume of spaces of matrices satisfying an approximate commutation was employed by von Neumann in 1942 \cite{vonNeumann1942}. Voiculescu's first von Neumann algebraic application of free entropy techniques in \cite{VoiculescuFE3} was based on the same principle, together with covering estimates for spaces of projections (or Grassmannians) proved by Szarek \cite{Szarek}.  These have become standard techniques in the study of microstates free entropy, free entropy dimension, and $1$-bounded entropy;  more generally, one can leverage not only commutation, but other properties of the bimodule structure of $L^2(M)$ over a subalgebra $A$; see e.g. \cite{Jung2007S1B,DimaSDE,Shl2015,Hayes2018}.  The idea is that volume of matrix microstate spaces in a neighborhood of a point can be analyzed using the approximate kernel or approximate image of some left-right multiplication operator $T(x) = \sum_{j=1}^k a_kxb_k$, and the projection-paving operator $\sum_{j=1}^k p_j x p_j$ used here is an instance of this.
	
	Instead of using the data processing inequality, we will here take a different approach that demonstrates a stronger claim about Wasserstein distance.  Namely, we will give a lower bound for the expectations
	\begin{equation} \label{eq: intro rough lower bound}
		\mathbb{E} \psi_k^{(n)}(Y^{(n)}) - \mathbb{E} \psi_k^{(n)}(S^{(n)}) \gtrsim (m+1)^{1/2} e^{-t/2} - \operatorname{Error}.
	\end{equation}
	Since $\psi_k^{(n)}$ is $1$-Lipschitz, this in turn produces a lower bound for the $L^2$ (or even $L^1$) Wasserstein distance $W_2(Y^{(n)}(t),S^{(n)})$.  In turn, the Wasserstein distance gives a lower bound for the square root of the classical Fisher information (relative to Gaussian) using the information-cost inequality (see Theorem \ref{thm: Gaussian functional inequalities} below).  We remark that the log-Sobolev inequality would also allow us to lower-bound the Gaussian Fisher information by the Gaussian entropy.
	
	This means that our approximate upper bound on the Gaussian free Fisher information is $(m+1)^2 e^{-4t}$ while the approximate lower bound on the Gaussian classical Fisher information is $(m+1)^{1/2} e^{-t/2}$.  Thus, the free Fisher information decays at a faster exponential rate than the classical Fisher information, provided that $t$ is not so large that the error term in \eqref{eq: intro rough lower bound} dominates, which requires careful estimates.
	
	The study of higher-order Stein kernels in \cite{Fathi2019Stein} gives a heuristic explanation for these differing rates.  If a random variable has moments up to order $k$ that agree with the Gaussian and the distribution is sufficiently regular, one can expect the Wasserstein distance of the Langevin diffusion to be bounded by a constant times $e^{-(k+1)t/2}$ as $t \to \infty$.  For the ideal random variables $X(t)$ in our example, the moments up to order $3$ agree with those of a free semicircular family, hence the Wasserstein distance decays as $e^{-2t}$ (as it is bounded by the square root of the Fisher information \cite{HU2006}).  However, if we consider random matrix models that have block structure relative to a given family of projections $P_j$ (pretending that the $p_j$'s are constant), the covariances of the matrix entries do not agree with those of the GUE ensemble, so the Wasserstein distance cannot decay faster than $e^{-t}$ (in this discussion, I focus only on the decay in $t$ and disregard factors of $m$).  However, the block structure of the random matrix ensembles cannot be detected by the noncommutative moments alone.
	
	However, there is significant subtlety here.  We proved the lower bound $e^{-t/2}$ holds for $t$ up to a certain threshold.  Based on a constant block structure, if the ensemble has mean zero, one would expect the decay rate to be $e^{-t}$ times the differences in variance with the Gaussian, for sufficiently large $t$ and up to some constants depending on $m$ and possibly $n$.  In fact, if we arrange that the distribution of $Y^{(n)}(t)$ is unitarily invariant, it is not immediately clear how to distinguish its covariances from those of GUE tuple.  However, even though $p_j$ is not literally constant, the number of degrees of freedom in choosing $p_j$ is a constant times $n^2$ (independent of $m$); thus, the possible choices of $p_j$ range in a much lower-dimensional space, so as long as $m$ is large and $t$ is not too large, we can use concentration of measure to obtain behavior similar to the constant-$p_j$ case.  The comparison of different decay rates also suggests an interesting question for further research, namely whether there is some phase transition in the behavior of $W_2(Y^{(n)}(t),S^{(n)})$ for large $t$.

	The ideas here also provide additional motivation for the continuous model theory framework for free entropy and large matrix models in \cite{Jekel2025InfoGeom,Jekel2026FIG}.  The method of proof illustrates the important role of variational formulas with suprema and infima.  The test function $\psi_k^{(n)}$ is defined using a supremum over a family of projections $p$, and this indeed corresponds to a definable predicate in continuous model theory, because the set of projection families $(p_j)_{j=1}^k$ with trace $1/k$ and $\sum_{j=1}^k p_j = 1$ is a definable.  We thus obtain a lower bound for the Wasserstein distance using a sup-formula.  And by contrast, it would not be possible to obtain this lower bound using any dimension-independent test function that only depends on the noncommutative moments (i.e.\ a quantifier-free formula).  Indeed, since the abstract free Wasserstein distance of $Y(t)$ and $S$ is bounded by the square root of the free Fisher information $\Phi_g^*(Y(t))$ by \cite{HU2006}, the difference $|\tau(f(Y(t))) - \tau(f(S))|$ (which also agrees with the large-$n$ limit of $|\mathbb{E} \tr_n(f(Y^{(n)}(t))) - \mathbb{E} \tr_n(f(S^{(n)}))|$) will be bounded by this free Wasserstein distance times the Lipschitz norm of $f$ on an appropriate operator-norm ball, for any noncommutative polynomial $f$.
	
	\subsection{Organization}
	
	The rest of the paper is organized as follows:
	\begin{itemize}
		\item \S \ref{sec: preliminaries} reviews prerequisites on free probability as well as classical and free entropy.
		\item \S \ref{sec: main upper bound} proves the upper bound for free Fisher information in terms of cumulants in general and estimates the cumulants in the main example.
		\item \S \ref{sec: block detection} studies the function $\psi$ that detects approximate block structure and shows a lower bound on the Wasserstein distance of the main example from Gaussian.
		\item \S \ref{sec: conclusion} concludes the proofs of the main results.
	\end{itemize}
	
	\subsection*{Acknowledgements and AI Declaration}
	
	This paper is of course inspired by the preprint generated and posted online by OpenAI.  I am also grateful to Jonathan Shi for sharing a Claude-generated assessment of missing attributions in OpenAI's preprint, which in particular pointed out the relevance of literature on Stein kernels.  However, other than working from these sources, I did not use any AI models in any way for the preparation of this paper.  I thank Charles-Philippe Diez for advice on free and classical Stein kernels.  I thank Diez and F{\'e}lix Parraud for finding typos in an earlier draft.
	
	\section{Preliminaries} \label{sec: preliminaries}
	
	\subsection{Free probability and random matrices}
	
	We assume familiarity with noncommutative probability spaces as presented in \cite[\S 5.2]{AGZ2009}. We recall the following definitions, conventions, and facts.
	
	\begin{notation}[Noncommutative probability spaces and $L^p$ norms] ~
		\begin{enumerate}[(1)]
			\item The noncommutative analog of a probability space used in this paper is a \emph{tracial von Neumann algebra}, namely a von Neumann algebra $M$ with a faithful normal tracial state $\tau: M \to \mathbb{C}$.  However, if the reader wants to minimize technical background, a $\mathrm{C}^*$-algebraic noncommutative probability space is also sufficient.
			\item For $X \in M_{\sa}^m$, we denote by $\mathrm{W}^*(X)$ the von Neumann subalgebra generated by $X$.
			\item $M_{\sa}$ denotes the set of self-adjoint elements in $M$.
			\item For tuples $X = (X_1,\dots,X_m)$ and $Y = (Y_1,\dots,Y_m)$ from $M$, the inner product with respect to the trace is given by
			\begin{equation}
			\ip{X,Y}_{\tau} = \sum_{j=1}^m \tau(X_j^*Y_j).
			\end{equation}
			The completion of $M$ with respect to this inner product (also known as the GNS construction) is denoted $L^2(M,\tau)$.
			\item Likewise, the $L^p$ norm with respect to the trace is
			\begin{equation}
			\norm{X}_p = \left( \sum_{j=1}^m \tau(|X_j|^p)\right)^{1/p} \text{ for } p \in [1,\infty),
			\end{equation}
			where $|X_j| = (X_j^*X_j)^{1/2}$, and
			\begin{equation}
			\norm{X}_\infty = \max(\norm{X_1},\dots,\norm{X_m}),
			\end{equation}
			where $\norm{\cdot}$ denotes the operator norm.  It is well known that for $X$, $Y \in M$, we have the noncommutative H{\"o}lder's inequality
			\begin{equation}
			\norm{XY}_p \leq \norm{X}_{p_1} \norm{Y}_{p_2}, \text{ when } \frac{1}{p_1} + \frac{1}{p_2} = \frac{1}{p},
			\end{equation}
			but we will only need this for the cases of $1$, $2$, and $\infty$.
		\end{enumerate}
	\end{notation}
	
	\begin{notation}[Noncommutative polynomials and laws] ~
	\begin{enumerate}[(1)]
		\item We denote by $\C\ip{x_1,\dots,x_m}$ the noncommutative polynomial algebra on $m$ variables, equipped with the unique $*$-algebra structure for which $x_j^* = x_j$.
		\item Given a self-adjoint tuple $X = (X_1,\dots,X_m)$ in a tracial von Neumann algebra $(M,\tau)$, the \emph{evaluation} $f(X)$ of a noncommutative polynomial $f$ is given by the unique homomorphism $\C\ip{x_1,\dots,x_m} \to M$ sending $x_j$ to $X_j$.
		\item The \emph{noncommutative law} of $X$ is the linear functional $\lambda_X: \C\ip{x_1,\dots,x_m} \to \C$ given by
		\begin{equation}
		\lambda_X(f) = \tau(f(X)).
		\end{equation}
		\item Given tracial von Neumann algebras $(M^{(n)},\tau^{(n)})$ for $n \in \N$ and $(M,\tau)$, given $X^{(n)} \in (M^{(n)})_{\sa}^m$ and $X \in M_{\sa}^m$, we say that $X^{(n)}$ \emph{converges to $X$ in noncommutative law} if
		\begin{equation}
		\lim_{n \to \infty} \tau^{(n)}(f(X^{(n)})) = \tau(f(X)) \text{ for all } f \in \C\ip{x_1,\dots,x_m}.
		\end{equation}
	\end{enumerate}
	\end{notation}
	
	\begin{notation}[Matrix algebras] \label{not: matrix algebras}
	$\mathbb{M}_n$ denotes the $n \times n$ complex matrix algebra and $\tr_n$ denotes the normalized trace $\tr_n = (1/n) \Tr_n$.  Note that $(\mathbb{M}_n,\tr_n)$ is an example of a tracial von Neumann algebra.
	
	We view $(\mathbb{M}_n)^m$ as a complex inner product space with respect to $\ip{\cdot,\cdot}_{\tr_n}$.  Note that this restricts to a real inner product on $(\mathbb{M}_n)_{\sa}^m$.  As a real inner product space, it is isomorphic to $\R^{mn^2}$, and we use this identification to define Lebesgue measure on $(\mathbb{M}_n)_{\sa}^m$ (as well as other concepts such as the gradient of a function).
	
	It is important to note that the resulting normalizations are different from those obtained by looking at matrices entrywise; namely, for us an orthonormal basis for $\mathbb{M}_n$ is given by
	\begin{equation}
	\{\sqrt{n} E_{j,k} \}_{j,k=1}^n
	\end{equation}
	or by
	\begin{equation}
	\{ \sqrt{n} E_{j,j} \}_{j=1}^n \cup \{\sqrt{n/2} (E_{j,k} + E_{k,j}) \}_{j < k} \cup \{\sqrt{n/2}(iE_{j,k} - i E_{j,k})\}_{j < k}.
	\end{equation}
	\end{notation}

	We recall the definition of free independence, GUE random matrix models, and Voiculescu's asymptotic freeness theorem.
	
	\begin{definition}[Free independence; see {\cite{VDN1992}}]
	Let $(M,\tau)$ be a tracial von Neumann algebra.
	\begin{enumerate}[(1)]
		\item $*$-subalgebras $A_1$, \dots, $A_m \subseteq M$ are \emph{freely independent} if for any $\ell \in \N$, for any $i_1$, \dots, $i_\ell \in [m]$ with $i_1 \neq i_2 \neq \dots \neq i_\ell$, for any $x_1 \in A_{i_1}$, \dots, $x_\ell \in A_{i_\ell}$, we have
		\begin{equation}
		\tau[(x_1 - \tau(x_1)) \dots (x_\ell - \tau(x_\ell))] = 0.
		\end{equation}
		\item Families of random variables $X_1 = (X_{1,1},\dots,X_{1,t_1})$, \dots, $X_m = (X_{m,1},\dots,X_{m,t_m})$ are said to be freely independent if the $*$-subalgebras generated by $X_1$, \dots, $X_m$ respectively are freely independent.
	\end{enumerate}
	\end{definition}
	
	We recall that given any list of tracial von Neumann algebras, one can form their free product, which contains canonical freely independent copies of the original algebras; see \cite{VDN1992}.
	
	\begin{definition}[Free semicircular family]
		A \emph{standard semicircular operator} is an element $X$ in a tracial von Neumann algebra $(M,\tau)$ with spectral distribution given by Wigner's semicircular measure, i.e., such that
		\begin{equation}
		\tau(f(X)) = \frac{1}{2\pi} \int_{-2}^2 f(t) \sqrt{4 - t^2}\,dt \text{ for } f \in \C[x].
		\end{equation}
		A \emph{standard semicircular family (or $m$-tuple)} is a family $(X_1,\dots,X_m)$ of standard semicircular operators such that $X_1$, \dots, $X_m$ are freely independent.
	\end{definition}
	
	\begin{notation}[GUE matrix tuple] \label{not: GUE tuple}
	A \emph{GUE matrix} is a random matrix $S^{(n)}$ with probability distribution $\gamma_1^{(n)}$ given by
	\begin{equation}
	d\gamma_1^{(n)}(x) = (2 \pi n^2)^{-n^2/2} \exp(-n^2 \tr_n(x^2))\,dx.
	\end{equation}
	Here $dx$ is Lebesgue measure on $(\mathbb{M}_n)_{\sa}$ (with the normalization as in Notation \ref{not: matrix algebras}).  The matrix $S^{(n)}$ corresponds to $1/n^2$ times a standard Gaussian vector in the inner product space $((\mathbb{M}_n)_{\sa},\ip{\cdot,\cdot}_{\tr_n})$.
	
	Similarly, a \emph{GUE matrix tuple} is a family $(S_1^{(n)},\dots,S_m^{(n)})$ of independent GUE matrices, and their probability distribution is denoted $\gamma_m^{(n)}$. 
	\end{notation}
	
	\begin{notation}[Convergence of random matrix models] \label{not: random matrix convergence}
	Let $X = (X_1,\dots,X_m)$ be a tuple of self-adjoint random variables in a tracial von Neumann algebra.  For $n \in \N$, let $X^{(n)} = (X_1^{(n)},\dots,X_m^{(n)})$ be a tuple of $n \times n$ random self-adjoint matrices, or random variables in $(\mathbb{M}_n)_{\sa}^m$.  We say that \emph{$X^{(n)}$ converges to $X$ in noncommutative law in probability} if
	\begin{equation}
	\lim_{n \to \infty} \tr_n(f(X^{(n)})) = \tr(f(X)) \text{ in probability for } f \in \C\ip{x_1,\dots,x_m}.
	\end{equation}
	\end{notation}
	
	\begin{theorem}[Voiculescu's asymptotic freeness; see {\cite[Theorem 2.2]{Voiculescu1991}, \cite[\S 5.4]{AGZ2009}}] \label{thm: asymptotic freeness}
	Let $X$ and $X^{(n)}$ be as in Notation \ref{not: random matrix convergence} with $X^{(n)}$ converging to $X$ in noncommutative law in probability, and assume that
	\begin{equation}
	\limsup_{n \to \infty} \norm{X^{(n)}}_\infty < \infty \text{ in probability.}
	\end{equation}
	Let $S^{(n)}$ be a GUE $m$-tuple independent of $X^{(n)}$, and let $S$ be a free semicircular family free from $X$.  Then $(X^{(n)},S^{(n)})$ converges to $(X,S)$ in noncommutative law in probability.
	\end{theorem}
	
	We also recall that GUE matrices satisfying strong tail bounds with respect to operator norm.
	
	\begin{theorem}[{\cite{Ledoux2003remark}}] \label{thm: GUE operator norm bound}
	Let $S^{(n)}$ be a GUE $m$-tuple.  Then for some positive constants $C$ and $c$,
	\begin{equation}
	\mathbb{P}(\norm{S^{(n)}} \geq 2 + \delta) \leq C m e^{-cn \delta^{3/2}}
	\end{equation}
	\end{theorem}
	
	See also \cite[\S 2.1.6, 2.6.2]{AGZ2009} for other estimates and proofs.  Note in particular that this implies by the Borel--Cantelli lemma that $\lim_{n \to \infty} \norm{S^{(n)}}_\infty = 2$ almost surely.  Moreover, $\mathbb{E} \exp(\norm{S}_\infty)$ is bounded by a constant independent of $n$.
	
	\subsection{Non-crossing partitions and cumulants}
	
	Next we recall the combinatorial machinery of non-crossing partitions, which is extremely useful for moment computations in free probability.  For background and details,  see \cite{Speicher1994,Speicher1998,NS2006}.
	
	\begin{notation}[Partitions] \label{not: partitions} ~
	\begin{enumerate}[(1)]
		\item We use the notation $[\ell] = \{1,\dots,\ell\}$.
		\item A \emph{partition} of $[\ell]$ is a collection $\pi$ of non-empty subsets of $[\ell]$, called \emph{blocks}, such that $[\ell] = \bigsqcup_{B \in \pi} B$.
		\item A \emph{pair partition} is one in which each block has exactly two elements.
		\item For a partition $\pi$ of $[\ell]$, a \emph{crossing} is a sequence of indices $i < i' < j < j'$ such that $i$ and $j$ are in some block $B$ and $i'$ and $j'$ are in some block $B' \neq B$.
		\item We say that $\pi$ is \emph{non-crossing} if it has no crossings.
		\item We denote by $\operatorname{NC}(\ell)$ the set of non-crossing partitions of $[\ell]$.
	\end{enumerate}
	\end{notation}
	
	\begin{notation}[Free cumulants] ~
		\begin{enumerate}[(1)]
			\item Given a vector space $V$, a sequence $\Lambda_k: V^k \to \C$ of multilinear maps, and a partition $\pi$ of $[\ell]$, let $\Lambda_\pi: V^\ell \to \mathbb{C}$ be given by
			\begin{equation}
			\Lambda_\pi(v_1,\dots,v_\ell) = \prod_{B \in \pi} \Lambda_{|B|}(v_j: j \in B),
			\end{equation}
			where $(v_j: j \in B)$ is understood to be ordered from left to right.
			\item For a noncommutative probability space $(M,\tau)$, the \emph{free cumulants} are the unique multilinear maps $K_k: M^k \to \C$ satisfying
			\begin{equation}
			\tau(X_1 \dots X_\ell) = \sum_{\pi \in \operatorname{NC}(\ell)} K_\pi(X_1,\dots,X_\ell).
			\end{equation}
			The existence and uniqueness of the free cumulants follows from M{\"o}bius inversion.
		\end{enumerate}
	\end{notation}
	
	Free independence can be expressed in terms of cumulants as follows.
	
	\begin{theorem}[{\cite[Theorem 11.15]{NS2006}}] \label{thm: freeness and cumulants}
	Let $A_1$, \dots, $A_m$ be $*$-subalgebras of $(M,\tau)$.  The following are equivalent:
	\begin{enumerate}[(1)]
		\item $A_1$, \dots, $A_m$ are freely independent.
		\item If $i_1$, \dots, $i_\ell \in [m]$ are not all equal to each other, and $X_1 \in A_{i_1}$, \dots, $X_\ell \in A_{i_\ell}$, then
		\begin{equation}
		K_\ell(X_1,\dots,X_\ell) = 0.
		\end{equation}
	\end{enumerate}
	\end{theorem}
	
	\begin{proposition} \label{prop: cumulants of semicircular}
	Let $(S_1,\dots,S_m)$ be a standard free semicircular family.  Let $i_1$, \dots, $i_\ell \in [m]$.  Then $K(X_{i_1},\dots,X_{i_\ell}) = \delta_{\ell=2, i_1=i_2}$.
	\end{proposition}
	
	\subsection{Classical entropy, Fisher information, and Wasserstein distance}
	
	Before delving in to the definition of free entropy, we recall some background on classical entropy and the dimensionally normalized version associated to $n \times n$ matrix tuples.  The material in this section is classical and well known.  Some of the exposition overlaps with \cite[\S 3]{JekelPi2024}, to which we refer for additional detail and citations.
	
	\begin{definition}[Differential entropy]
		Let $\mu$ be a probability measure on $\R^m$ with absolutely continuous density $\rho$ with respect to Lebesgue measure, i.e. $d\mu = \rho \, dx$. The \textit{entropy} of $\mu$ is defined to be
		\begin{equation}
		h(\mu) := \int_{\R^m} - \rho \log \rho \, dx,
		\end{equation}
		whenever the integral is defined. If $\mu$ does not have a density $\rho$, then we set $h(\mu) := - \infty$. 
		If $X$ is a random variable with distribution $\mu$, then we also write $h(X)$ in place of $h(\mu)$.
	\end{definition}
	
	\begin{definition}[Fisher information]
		Suppose that $\mu$ is a probability measure on $\R^m$ with $C^1$ density $\rho$.  The \emph{Fisher information} of $\mu$ is defined as
		\begin{equation}
		\mathcal{I}(\mu) := \int_{\R^m} |\nabla \rho/ \rho|^2\,d\mu.
		\end{equation}
		If $X \sim \mu$ is a random variable and $\Xi = -(\nabla \rho / \rho)(X)$ is called its \emph{score function} and
		\begin{equation}
		\mathcal{I}(X) = \mathbb{E} |\Xi|^2.
		\end{equation}
	\end{definition}
	
	\begin{proposition}[de Bruijn identity]
		Let $X$ be a random variable in $\R^m$ with finite second moment and distribution $\mu$.  Let $X_t = X + Z_t$ where $Z_t$ is a standard Brownian motion on $\R^m$ (equivalently, its distribution $\mu_t$ is the convolution of $\mu$ with the Gaussian measure of covariance $tI$).  Then $X_t$ has a smooth density, and
		\begin{equation} \label{eq: de Bruijn finite}
		\frac{d}{dt} h(X_t) = \frac{1}{2} \mathcal{I}(X_t) \text{ for } t > 0.
		\end{equation}
		Moreover,
		\begin{equation} \label{eq: de Bruijn infinite}
			h(X) = \frac{1}{2} \int_0^\infty \left[ \frac{m}{1+t} - \mathcal{I}(X_t) \right] \,dt + \frac{m}{2} \log(2\pi e).
		\end{equation}
	\end{proposition}
	
	Although $\chi$ and $\chi^*$ are formulated to be analogs of entropy $h$ with respect to Lebesgue measure, it will also be convenient to use entropy with respect to another background probability measure, known as \emph{relative entropy} or \emph{KL-divergence}.  Due to the mixture of sign conventions in the literature, we will use the opposite sign here to arrange that the relative entropy is nonnegative (this well-known fact follows from Jensen's inequality).
	
	\begin{definition}
		For probability measures $\mu$ and $\nu$ on $\R^m$, let
		\begin{equation}
		h(\mu || \nu) = \int_{\R^m} \rho \log \rho \, d\nu \text{ when } d\mu = \rho\,d\nu,
		\end{equation}
		and $+\infty$ when $\mu$ does not have a density with respect to $\nu$.
	\end{definition}
	
	\begin{observation}
		Let $\gamma$ be the standard Gaussian measure on $\R^m$ of covariance $I$.  Let $\mu$ be a probability measure on $\R^m$.  Then
		\begin{equation}
		h(\mu || \gamma) = \frac{1}{2} \int_{\R^m} |x|^2\,d\mu(x) - h(\mu) + \frac{m}{2} \log(2\pi).
		\end{equation}
		We will also denote $h(\mu || \gamma)$ as $h_g(\mu)$ for short (`g' indicating ``Gaussian'').
	\end{observation}
	
	\begin{definition}
		For a probability measure $\mu$ with $C^1$ density $\rho$ with respect to Gaussian measure, let
		\begin{equation}
		\mathcal{I}_g(\mu) := \int_{\R^m} |\nabla \rho / \rho|^2\,d\mu.
		\end{equation}
	\end{definition}
	
	\begin{observation}
		If $\widehat{\rho}$ is the density with respect to Lebesgue measure, $X \sim \mu$, and $\Xi = -(\nabla \widehat{\rho} / \widehat{\rho})(X)$ is the score function, $\log \rho(x) = \log \widehat{\rho}(x) + \frac{1}{2} |x|^2 + \text{constant}$, and so
		\begin{equation}
		\mathcal{I}_g(\mu) = \mathcal{I}_g(X) = \mathbb{E} |\Xi - X|^2 = \mathbb{E} |\Xi|^2 - 2m + \mathbb{E} |X|^2,
		\end{equation}
		where we have applied the identity $\mathbb{E} \ip{\Xi, X} = m$ from integration by parts.
	\end{observation}
	
	There is a corresponding de Bruijn identity for entropy with respect to Gaussian (and for relative entropy more generally).
	
	\begin{proposition}
		Fix $\mu \in \mathcal{P}(\mathbb{R}^m)$ with finite second moment and $X \sim \mu$.  Let $Z \sim \gamma$ be an independent Gaussian and
		\begin{equation} \label{eq: Gaussian Langevin Xt}
		X_t = e^{-t/2} X + (1 - e^{-t})^{1/2}Z.
		\end{equation}
		Then
		\begin{equation}
			\frac{d}{dt} h_g(X_t) = -\frac{1}{2} \mathcal{I}_g(X_t).
		\end{equation}
		and
		\begin{equation} \label{eq: de Bruijn infinite Gaussian}
			h_g(X) = \frac{1}{2} \int_0^\infty \mathcal{I}_g(X_t).
		\end{equation}
	\end{proposition}
	
	\begin{observation} \label{obs: entropy conversion}
	In fact, \eqref{eq: de Bruijn infinite Gaussian} can be deduced from \eqref{eq: de Bruijn infinite} by a change of variables.  Indeed, if $X_t$ is given by \eqref{eq: Gaussian Langevin Xt}, then
	\begin{equation}
		X_t = e^{-t/2} (X + (e^t - 1)^{1/2} Z),
	\end{equation}
	and using a change of variables in the integral to define to Fisher information
	\begin{align}
		\mathcal{I}(X + (e^t-1)^{1/2}Z) &= \mathcal{I}(e^{t/2} X_t) \\
		&= e^{-t} \mathcal{I}(X_t) \\
		&= e^{-t} (\mathcal{I}_g(X_t) + 2m - \mathbb{E} |X_t|^2 ) \\
		&= e^{-t} (\mathcal{I}_g(X_t) + 2m - e^{-t} \mathbb{E} |X|^2 - (1 - e^{-t}) \mathbb{E}|Z|^2) \\
		&= e^{-t} (\mathcal{I}_g(X_t) + m(1 + e^{-t}) - e^{-t} \mathbb{E} |X|^2),
	\end{align}
	so taking $u = e^t - 1$,
	\begin{align}
		\frac{1}{2} \int_0^\infty \left[ \frac{m}{u+1} - \mathcal{I}(X + u^{1/2}Z) \right] &=
		\frac{1}{2} \int_0^\infty \left[ m e^{-t} - \mathcal{I}(X + (e^t - 1)^{1/2}Z) \right] e^t\,dt \\
		&= -\frac{1}{2} \int_0^\infty \mathcal{I}_g(X_t) \,dt - \frac{1}{2} \int_0^\infty m e^{-t}\,dt + \frac{1}{2} \int_0^\infty e^{-t} E|X|^2\,dt \\
		&= \frac{1}{2} \mathbb{E}|X|^2 - \frac{1}{2} \int_0^\infty \mathcal{I}_g(X_t) \,dt - \frac{m}{2}.
	\end{align}
	From this, we see that \eqref{eq: de Bruijn infinite} is equivalent to \eqref{eq: de Bruijn infinite Gaussian}.
	\end{observation}
	
	Finally, we recall the $L^2$-Wasserstein distance (for further background, see e.g.\ \cite{Villani2008}).
	
	\begin{definition}
		For probability measures $\mu$, $\nu$ on $\R^m$, the $L^2$-Wasserstein distance is defined as
		\begin{equation} \label{eq: definition of W2}
			W_2(\mu,\nu) = \inf \{ \norm{X - Y}_{L^2}: X \sim \mu, Y \sim \nu \},
		\end{equation}
		where we consider all possible pairs of random variables $(X,Y)$ on some atomless probability space with distributions $\mu$ and $\nu$ respectively.  We also write $W_2(X,Y)$ for the Wasserstein distance of the distributions of $X$ and $Y$.
	\end{definition}

	The Gaussian measure $\gamma$ satisfies important functional inequalities relating the entropy, Wasserstein distance, and Fisher information.  Note that the normalization of Fisher information here differs by a factor of $2$ from some sources.
	
	\begin{theorem} \label{thm: Gaussian functional inequalities}
		For random variables $X$ in $\R^m$ with $C^1$ density, we have the following inequalities:
		\begin{enumerate}[(1)]
			\item \emph{log-Sobolev inequality:} $h_g(X) \leq \frac{1}{2} \mathcal{I}_g(X)$ \cite{Bakry1985}, \cite{BoL2000}, \cite[\S 4.4.2]{AGZ2009}.
			\item \emph{Talagrand inequality:} $\frac{1}{2} W_2(X,Z)^2 \leq h_g(X)$ where $Z \sim \gamma$ \cite{Talagrand1996}, \cite{OttoVillani2000}.
			\item \emph{Information-cost inequality:} $W_2(X,Z) \leq \mathcal{I}_g(X)^{1/2}$.
		\end{enumerate}
		Here item (3) follows immediately from the combination of the two previous items.
	\end{theorem}
	
	We close this subsection with the normalized versions of the entropy and Fisher information for the random matrix context.  Recall the GUE measure $\gamma_m^{(n)}$ defined in Notation \ref{not: GUE tuple}, which is the $1/n^2$ rescaling of the standard Gaussian measure on $((\mathbb{M}_n)_{\sa}^m, \ip{\cdot,\cdot}_{\tr_n})$.
	
	\begin{notation} \label{not: normalized entropy and Fisher}
	Define the following normalized quantities for random variables $X$ in $(\mathbb{M}_n)_{\sa}^m$:
	\begin{align}
		h^{(n)}(X) &= \frac{1}{n^2} h(X) + m \log n \\
		h_g^{(n)}(X) &= \frac{1}{n^2} h_g(X) - m \log n \\
		\mathcal{I}^{(n)}(X) &= \frac{1}{n^4} \mathcal{I}(X) \\
		\mathcal{I}_g^{(n)}(X) &= \frac{1}{n^2} \mathcal{I}_g(X).
	\end{align}
	\end{notation}
	
	\begin{proposition}
		Let $X$ be a random variable in $(\mathbb{M}_n)_{\sa}^m$ with finite variance and let $Z$ be an independent GUE $m$-tuple.
		\begin{align}
			\frac{d}{dt} h^{(n)}(X + t^{1/2}Z) &= \frac{1}{2} \mathcal{I}^{(n)}(X + t^{1/2} Z) \text{ for } t > 0 \\
			\frac{d}{dt} h_g^{(n)}(e^{-t/2}X + (1 - e^{-t})^{1/2} Z) &= -\frac{1}{2} \mathcal{I}_g^{(n)}(e^{-t/2}X + (1 - e^{-t})^{1/2} Z) \text{ for } t > 0. \\
			\mathcal{I}_g^{(n)}(X) &= \mathcal{I}(X) - 2m + \mathbb{E} \norm{X}_2^2.
		\end{align}
		Moreover, we have
		\begin{align}
			\frac{1}{2} W_2(X,Z)^2 &\leq h_g^{(n)}(X) \\
			h_g^{(n)}(X) &\leq \frac{1}{2} \mathcal{I}_g^{(n)}(X) \\
			W_2(X,Z) &\leq \mathcal{I}_g^{(n)}(X)^{1/2}. \label{eq: normalized info cost}
		\end{align}
	\end{proposition}
	
	\subsection{Free entropy}
	
	This section will review the free entropies $\chi$ and $\chi^*$.  The free entropy $\chi$ is defined in terms of volumes of spaces of matrix approximations \cite{VoiculescuFE2}.  However, we will not need to use the definition directly, and we can instead use the following characterization as the largest possible limit of entropies of matrix models.
	
	\begin{proposition}[{\cite[Proposition B.7]{ST2022fractional}, \cite[Theorem 4.10]{JekelPi2024}}] \label{prop: variational formula}
		Let $X^{(n)}$ be an $m$-tuple of self-adjoint random matrices, and let $X$ be an $m$-tuple of self-adjoint elements in a tracial von Neumann algebra $(M,\tau)$.  Assume that
		\begin{enumerate}[(1)]
			\item $X^{(n)}$ converges in noncommutative law in probability to $X$.
			\item $\limsup_{n \to \infty} \norm{X^{(n)}}_\infty \leq R$ for some constant $R$.
			\item There are some positive constants $c$, $C$, $K$, such that
			\begin{equation} \label{eq: 2norm tail bound hypothesis}
				P(\norm{X^{(n)}}_2 \geq C + \delta) \leq K e^{-cn^2 \delta^2}
			\end{equation}
		\end{enumerate}
		Then we have
		\begin{equation} \label{eq: limit of classical entropy bound}
			\limsup_{n \to \infty} h^{(n)}(X^{(n)}) \leq \chi(X).
		\end{equation}
		Conversely, if $\norm{X}_\infty < R$ and $\mathrm{W}^*(X)$ is Connes-embeddable, then there exist random matrix models with $\norm{X^{(n)}} \leq R$  $X^{(n)}$ converging to $X$ in noncommutative law such that
		\begin{equation} \label{eq: limit of classical entropy}
			\limsup_{n \to \infty} h^{(n)}(X) = \chi(X).
		\end{equation}
	\end{proposition}
	
	\begin{remark}
		Note that \cite[Theorem 4.10]{JekelPi2024} is stated for limits along ultrafilters, but it also holds for $\limsup$ since the $\limsup$ is the supremum over all ultrafilters.  Although the constant $c$ in item (3) is not included in \cite[Theorem 4.15]{JekelPi2024}, the same proof applies with only cosmetic changes.
	\end{remark}
	
	Next, we recall the free Fisher information $\Phi^*$ and entropy $\chi^*$ from \cite{VoiculescuFE5}.
	
	\begin{definition}[See {\cite[\S 3]{VoiculescuFE5}}]
		Fix $m \in \N$.  The \emph{free difference quotient} $\partial_j: \C\ip{x_1,\dots,x_m} \to \C\ip{x_1,\dots,x_m} \otimes \C\ip{x_1,\dots,x_m}$ is the unique linear map satisfying
		\begin{align}
			\partial_j(x_k) &= \delta_{j,k} 1 \otimes 1 \\
			\partial_j(fg) &= f \cdot \partial_j g + \partial_j f \cdot g,
		\end{align}
		where $f \cdot \partial_j g = (f \otimes 1) \partial_j g$ represents left multiplication on the left tensorand and $\partial_j f \cdot g$ represents right multiplication on the right tensorand.
	\end{definition}
	
	The free analog of a score function is defined via an integration-by-parts relation analogous to the ones satisfied by the classical score function $(-\nabla \rho / \rho)(X)$.
	
	\begin{definition}[See {\cite[\S 3]{VoiculescuFE5}}]
		Let $X$ be a self-adjoint $m$-tuple from $(M,\tau)$.  We say that $\xi_j \in L^2(\mathrm{W}^*(X),\tau)$ is a \emph{$j$th conjugate variable} or \emph{free score function} if we have
		\begin{equation} \label{eq: conjugate variable def}
			\ip{\xi_j, f(X)}_\tau = \tau \otimes \tau[\partial_j f(X)] \text{ for } f \in \C\ip{x_1,\dots,x_m}.
		\end{equation}
		If conjugate variables $\xi_j$ exist for all $j = 1$, \dots, $m$, we define the \emph{free Fisher information}
		\begin{equation}
			\Phi^*(X) := \norm{\xi}_2^2.
		\end{equation}
		If a family of conjugate variables does not exist, then set $\Phi^*(X) = +\infty$.
	\end{definition}
	
	The non-microstates free entropy $\chi^*$ is defined via the analog of the de Bruijn identity.
	
	\begin{definition}[See {\cite[\S 7]{VoiculescuFE5}}]
		Let $X$ be a self-adjoint $m$-tuple from a tracial von Neumann algebra $(M,\tau)$.  Let $Z$ be a freely independent semicircular family (we can assume $M$ contains such a $Z$ by taking a free product if needed).  Then let
		\begin{equation}
			\chi^*(X) = \frac{1}{2} \int_0^\infty \frac{m}{1+t} - \Phi^*(X + t^{1/2}S)\,dt + \frac{m}{2} \log(2\pi e).
		\end{equation}
	\end{definition}
	
	This definition makes sense as \cite{VoiculescuFE5} shows that $\Phi^*(X + t^{1/2}S)$ is finite for $t > 0$ and the positive part of the integral above is finite.  It also follows from a changes of variables argument and basic manipulations that
	\begin{equation}
		\chi^*(X + t^{1/2} S) = \frac{1}{2} \int_t^\infty \left[\frac{m}{1 + s} - \Phi^*(X + s^{1/2} S) \right]\,ds + \frac{m}{2} \log(1+t) + \frac{m}{2} \log(2\pi e).
	\end{equation}
	From this we deduce by Lebesgue differentiation that
	\begin{equation}
		\frac{d}{dt} \chi^*(X + t^{1/2}S) = \frac{1}{2} \Phi^*(X + t^{1/2} S) \text{ for a.e. } t > 0.
	\end{equation}
	
	The Gaussian versions of $\chi^*$ and $\Phi^*$ are defined as follows, analogous to the classical case.
	\begin{definition} \label{def: Gaussian free Fisher}
		For a self-adjoint $m$-tuple $X$ from a tracial von Neumann algebra, define
		\begin{align}
			\chi_g^*(X) &= \frac{1}{2} \norm{X}_2^2 - \chi^*(X) + \frac{m}{2} \log(2\pi). \\
			\Phi_g^*(X) &= \Phi^*(X) - 2m + \norm{X}_2^2 = \norm{\xi - X}_2^2,
		\end{align}
		where the last equality holds under the assumption that the conjugate variables $\xi$ for $X$ exist (otherwise $\Phi_g^*(X)$ and $\Phi(X)$ are both $+\infty$).
	\end{definition}
	
	The analogous computations as Observation \ref{obs: entropy conversion} show the following.
	
	\begin{lemma} \label{lem: entropy conversion 2}
	For a self-adjoint $m$-tuple from a tracial von Neumann algebra and a freely independent standard semicircular family, we have
	\begin{equation}
		\chi_g^*(X) = \frac{1}{2} \int_0^\infty \Phi_g^*(e^{-t/2} X + (1 - e^{-t})^{1/2} Z).
	\end{equation}
	\end{lemma}
	
	We end this section with the bound relating $\chi$ and $\chi^*$.  Recall that $\chi \leq \chi^*$ was proved in general in \cite{BCG2003}.  Another proof was given in \cite{JekelPi2024}, from which we recall the following lemma relating free Fisher information and classical Fisher information of random matrix models.
	
	\begin{proposition}[{See \cite[Proposition 6.1]{JekelPi2024}}] \label{prop: classical versus free Fisher information}
		Let $X^{(n)}$ be an $m$-tuple of $n \times n$ self-adjoint random matrices and $X$ a self-adjoint $m$-tuple from a tracial von Neumann algebra $(M,\tau)$.   Assume that $X^{(n)}$ converges to $X$ in noncommutative law in probability, and furthermore assume that
		\begin{equation}  \label{eq: bounded moment assumption}
			\sup_n \mathbb{E} \norm{X}_p^p < \infty \text{ for } p \in [1,\infty).
		\end{equation}
		Then
		\begin{equation} \label{eq: classical versus free Fisher info}
			\liminf_{n \to \infty} \mathcal{I}^{(n)}(X^{(n)}) \geq \Phi^*(X).
		\end{equation}
	\end{proposition}
	
	\begin{observation} \label{obs: subtract converging second moment}
		Note \eqref{eq: classical versus free Fisher info} also holds for $\mathcal{I}_g^{(n)}(X^{(n)})$ and $\Phi_g^*(X)$.  Indeed, the Gaussian versions differ from the Lebesgue versions by subtracting $2m$ and adding $\mathbb{E} \norm{X^{(n)}}_2^2$ or $\norm{X}_2^2$ respectively.  But using a truncation or generalized dominated convergence argument, as in the proof of \cite[Proposition 6.1]{JekelPi2024} for instance,
		\begin{equation} \label{eq: strict Fisher comparison hypothesis}
			\lim_{n \to \infty} \mathbb{E} \norm{X^{(n)}}_2^2 = \norm{X}_2^2.
		\end{equation}
	\end{observation}
	
	This also leads to the following statement, which is a related result to $\chi \leq \chi^*$ and the proof given in \cite[Theorem 6.4]{JekelPi2024}, but looking at the integral of Fisher information over a finite interval rather than $[0,\infty)$.  This roughly corresponds to \cite[\S 5]{OpenAI2026entropy}.
	
	\begin{proposition} \label{prop: classical versus free entropy strict}
		Let $X$ be an $m$-tuple from a tracial von Neumann algebra with $\norm{X}_\infty < R$ and $\chi(X) > -\infty$, and let $S$ be a freely independent standard semicircular $m$-tuple.  Let $X^{(n)}$ be random matrix models for $X$ satisfying \eqref{eq: limit of classical entropy} given by the second part of Proposition \ref{prop: variational formula}.  Let $S^{(n)}$ be an independent GUE $m$-tuple.  Let $t_1 > 0$, and suppose that
		\begin{equation} \label{eq: strict Fisher bound}
			\liminf_{n \to \infty} \mathcal{I}^{(n)}(X^{(n)} + t^{1/2} S^{(n)}) > \Phi^*(X + t^{1/2}S)
		\end{equation}
		for $t$ in a set of positive measure in $[0,\infty)$.  Then
		\begin{equation}
			\chi(X) < \chi^*(X).
		\end{equation}
	\end{proposition}
	
	\begin{proof}
		Fix $t_1 > 0$ such that \eqref{eq: strict Fisher bound} holds for a set of positive measure in $[0,t_1]$.
		First, by Voiculescu's asymptotic freeness theorem (Theorem \ref{thm: asymptotic freeness}), $(X^{(n)},S^{(n)})$ converges to $(X,S)$ in noncommutative law in probability.  Hence, so does $X^{(n)} + t^{1/2} S^{(n)}$ to $X + t^{1/2} S$.  Moreover, note that the GUE tuple $S^{(n)}$ satisfies \eqref{eq: 2norm tail bound hypothesis} due to concentration of measure (see \cite[\S 2.3]{AGZ2009}) and $X^{(n)}$ trivially satisfies \eqref{eq: 2norm tail bound hypothesis} since it was chosen to be bounded in operator norm (per Proposition \ref{prop: variational formula}).  Hence, the first part of Proposition \ref{prop: variational formula} can be applied to $X^{(n)} + t^{1/2} S^{(n)}$, so by \eqref{eq: limit of classical entropy bound},
		\begin{equation}
			\limsup_{n \to \infty} h^{(n)}(X^{(n)} + t^{1/2} S^{(n)}) \leq \chi(X + t^{1/2} S).
		\end{equation}
		By Proposition \ref{prop: classical versus free Fisher information}, we have
		\begin{equation}
			\liminf_{n \to \infty} \mathcal{I}^{(n)}(X^{(n)} + t^{1/2} S^{(n)}) \geq \Phi^*(X + t^{1/2} S),
		\end{equation}
		and by hypothesis \eqref{eq: strict Fisher bound} strict inequality holds for a set of positive measure in $[t_0,t_1]$.  Hence, using Fatou's lemma,
		\begin{equation}
			\liminf_{n \to \infty} \int_0^{t_1} \mathcal{I}^{(n)}(X^{(n)} + t^{1/2} S^{(n)}) \,dt > \int_0^{t_1} \Phi^*(X + t^{1/2} S)\,dt.
		\end{equation}
		In other words,
		\begin{equation}
			\liminf_{n \to \infty} [h^{(n)}(X^{(n)} + t^{1/2} S^{(n)}) - h^{(n)}(X^{(n)})] > \chi^*(X + t^{1/2}S) - \chi^*(X).
		\end{equation}
		Hence, using $\chi \leq \chi^*$ and the choice of $X^{(n)}$,
		\begin{align}
			\chi(X) &= \limsup_{n \to \infty} h^{(n)}(X^{(n)}) \\
			&\leq \limsup_{n \to \infty} h^{(n)}(X^{(n)} + t^{1/2} S^{(n)}) - \liminf_{n \to \infty} [h^{(n)}(X^{(n)} + t^{1/2} S^{(n)}) - h^{(n)}(X^{(n)})] \\
			&< \chi(X + t^{1/2} S) - \chi^*(X + t^{1/2}S) + \chi^*(X) \\
			&\leq \chi^*(X),
		\end{align}
		which completes the proof.
	\end{proof}
	
	\section{Upper bound on free Fisher information} \label{sec: main upper bound}
	
	\subsection{General upper bound}
	
	\begin{notation}
		For a string $\vec{\imath} = (i_1,\dots,i_k) \in [m]^k$, define the iterated free difference quotient
		\begin{equation}
		\partial_{\vec{\imath}} f = (\partial_{i_1} \otimes \id^{\otimes (k-1)}) \circ \dots \circ \partial_{i_k} f.
		\end{equation}
		More explicitly, for a monomial $f(x) = x_{j_1} \dots x_{j_\ell}$, we have
		\begin{equation}
		\partial_{\vec{\imath}}[x_{j_1} \dots x_{j_\ell}] = \sum_{1 \leq t_1 < \dots < t_k \leq \ell} \mathbf{1}_{j_{t_1} = i_1, \dots, j_{t_k} = i_k} (x_{j_1} \dots x_{j_{t_1-1}}) \otimes (x_{j_{t_1+1}} \dots x_{j_{t_2-1}}) \otimes \dots \otimes (x_{j_{t_k+1}} \dots x_{j_\ell})
		\end{equation}
		We also recall the coassociativity relation
		\begin{equation} \label{eq: coassociativity}
			(\partial_{i_1} \otimes \id) \circ \partial_{i_2} = (\id \otimes \partial_{i_2}) \circ \partial_{i_1},
		\end{equation}
		which implies that
		\begin{equation} \label{eq: kth order derivative reversed}
			\partial_{\vec{\imath}} = (\id^{\otimes (k-1)} \otimes \partial_{i_k}) \circ \dots \circ \partial_{i_1}.
		\end{equation}
	\end{notation}
	
	\begin{definition}[{See \cite[Definition 3.1]{VoiculescuFE3}}]
		Fix a tuple $(X_1,\dots,X_m)$ of noncommutative random variables that generate a tracial von Neumann algebra $(M,\tau)$.  For $\vec{\imath} = (i_1,\dots,i_k) \in [m]^k$, we say that $\xi \in L^2(M,\tau)$ is a \emph{$\vec{\imath}$-conjugate variable} for $(X_1,\dots,X_m)$ if for all $f \in \C\ip{x}$, we have
		\begin{equation} \label{eq: higher order conjugate}
			\ip{\xi,f(X)} = \tau^{\otimes (k+1)}(\partial_{\vec{\imath}} f(X)).
		\end{equation}
	\end{definition}
	
	It will be important to have an explicit description of the higher-order conjugate variables for a free semicircular family.
	
	\begin{example} \label{ex: Chebyshev}
		For $(i_1,\dots,i_k) \in [m]^k$, consider the \emph{multivariate noncommutative Chebyshev polynomials} $U_{i_1,\dots,i_k}$ defined recursively as follows (here $\emptyset$ is the empty string):
		\begin{align}
			U_{\emptyset}(x) &= 1 \\
			U_i(x) &= x_i \\
			U_{i_1,\dots,i_{k+1}}(x) &= x_{i_1} U_{i_2,\dots,i_{k+1}}(x) - \delta_{i_1,i_2} U_{i_3,\dots,i_{k+1}} \text{ for } k \geq 1.
		\end{align}
		If $S = (S_1,\dots,S_m)$ is a standard semicircular family and $M$ is their von Neumann algebra, it is well known that $U_{i_1,\dots,i_k}(S)$ form an orthonormal basis for $L^2(M,\tau)$.  This is nothing other than the canonical orthonormal basis for the free Fock space (see \cite[\S 2]{BozejkoKummererSpeicher1997}, \cite{Ito2026}, \cite[Proposition 3.5]{Parraud2025}), or the Wigner chaos expansion (see \cite{BS1998,KNPS2012})
		
		The fact that $U_{\vec{\imath}}(S)$ is the $\vec{\imath}$-conjugate variable for $S$ has been observed before.  Voiculescu \cite[Proposition 3.8]{VoiculescuFE3} is the case when $i_1 = i_2 = \dots = i_k$, and the general case is implicit in \cite[Definition 15]{Diez2023} and explicitly written in \cite[Proposition 3.3]{Parraud2025}.  The inductive proof goes as follows.  The base case of the empty string $\emptyset$ is trivial and for strings of length $1$, it is well known that $S_i$ is the $i$-conjugate variable for $S$.  For the inductive step, assume the claim holds for strings up to length $k$.  Let $f$ be a polynomial and observe that
		\begin{align}
			\ip{S_{i_1} U_{i_2,\dots,i_{k+1}}(S), f(S)}_\tau &= \ip{U_{i_2,\dots,i_{k+1}}(S), S_{i_1} f(S)}_\tau \\
			&= \tau^{\otimes (k+1)}[\partial_{i_2,\dots,i_{k+1}} [S_{i_1} f(S)]] \\
			&= \tau^{\otimes (k+1)}[(\id \otimes \partial_{i_3,\dots,i_{k+1}}) \circ \partial_{i_2} [S_{i_1} f(S)]] \\
			&= \tau^{\otimes (k+1)}[(\id \otimes \partial_{i_3,\dots,i_{k+1}}) \circ [\delta_{i_1,i_2} 1 \otimes f(S) + S_{i_1} \partial_{i_2} f(S)]] \\
			&= \delta_{i_1,i_2} \tau^{\otimes k}[\partial_{i_3,\dots,i_{k+1}} f(S)] + \tau^{\otimes (k+1)}[S_{i_1} \partial_{i_2,\dots,i_{k+1}} f(S)] \\
			&= \delta_{i_1,i_2} \ip{U_{i_3,\dots,i_{k+1}}(S),f(S)} + \tau^{\otimes (k+2)}[\partial_{i_1,i_2,\dots,i_{k+1}} f(S)].
		\end{align}
		In the last line, we applied the conjugate-variable identity $\tau(S_{i_1} g(S)) = \tau \otimes \tau(\partial_{i_1} g(S))$ on the first tensorand of $\partial_{i_2,\dots,i_{k+1}} f(S)$.  Rearranging the last inequality produces
		\begin{align}
			\ip{U_{i_1,\dots,i_{k+1}}(S),f(S)}_\tau &= \ip{S_{i_1} U_{i_2,\dots,i_{k+1}}(S), f(S)}_\tau - \delta_{i_1,i_2} \ip{U_{i_3,\dots,i_{k+1}}, f(S)} \\
			&= \tau^{\otimes (k+2)}[\partial_{i_1,\dots,i_{k+1}} f(S)],
		\end{align}
		completing the induction.
	\end{example}
	
	The first idea implicit in the proof of OpenAI is the following noncommutative Stein identity.  This follows from the free moment-cumulant expansion and thus is likely folklore.  The one-variable classical analog of this identity is \cite[Corollary 1]{Barbour1986} \cite[Proposition 2.2]{NP2010cumulants}.  The finite-order expansion is similar to the iterated integration by parts used in the study of free Stein kernels \cite{Diez2023}, while the first-order case is closely related to the cumulant characterization of conjugate variables \cite[\S 8.2, Theorem 20]{MS2017}.
	
	\begin{proposition}[Stein identity] \label{prop: Stein identity}
		Let $X = (X_1,\dots,X_m)$ be a tuple of self-adjoint elements in a noncommutative probability space $(M,\tau)$.  For every noncommutative polynomial $f$ and $i \in [m]$, we have
		\begin{equation} \label{eq: Stein identity}
			\tau(X_i f(X)) = \sum_{k=0}^\infty \sum_{\vec{\imath} \in [m]^k} K_{k+1}(X_i, X_{i_1},\dots,X_{i_k}) \tau^{\otimes(k+1)} [\partial_{\vec{\imath}} f(X)].
		\end{equation}
		Here the sum has only finitely many terms since the derivatives of order higher than the degree of $f$ vanish.
	\end{proposition}
	
	\begin{proof}
		It suffices to prove the claim when $f$ is a monomial, say $f(x) = x_{j_1} \dots x_{j_\ell}$.  By the free moment-cumulant formula,
		\begin{equation}
		\tau(X_i f(X)) = \sum_{\pi \in \NC(\ell+1)} K_\pi(X_i,X_{j_1},\dots,X_{j_\ell}).
		\end{equation}
		For each partition $\pi$, let $B$ be the block containing $1$, and write $B$ as $\{1,t_1+1,\dots,t_k+1\}$ for some $1 \leq t_1 < \dots < t_k \leq \ell$.  Since $\pi$ is non-crossing, $\pi$ is the union of $B$ with the subpartitions between consecutive $t_j$'s, namely have
		\begin{equation}
		\pi = \{B\} \cup \pi_0 \cup \dots \cup \pi_k
		\end{equation}
		where $\pi_j$ is the restriction of $\pi$ to $\{t_j+2,\dots,t_j\}$, with the convention that $t_0 = 0$ and $t_k = \ell + 1$.  If we sum over all partitions $\pi$ with a fixed value of $B$, we can again apply the moment-cumulant formula within each interval $\{t_j+2,\dots,t_j\}$ to obtain
		\begin{equation}
		K_{k+1}(X_i,X_{j_{t_1}}, \dots, X_{j_{t_k}}) \tau(X_{j_1} \dots X_{j_{t_1-1}}) \tau(X_{j_{t_1+1}} \dots X_{j_{t_2-1}}) \dots \tau(X_{j_{t_k+1}} \dots X_{j_\ell})
		\end{equation}
		If we fix $k$ and sum over the possible values of $B$ with $|B| = k + 1$, this is equivalent to summing over $1 \leq t_1 < \dots < t_k \leq \ell$, which yields
		\begin{equation}
		\sum_{1 \leq t_1 < \dots < t_k \leq \ell} K_{k+1}(X_i,X_{j_{t_1}}, \dots, X_{j_{t_k}}) \tau(X_{j_1} \dots X_{j_{t_1-1}}) \tau(X_{j_{t_1+1}} \dots X_{j_{t_2-1}}) \dots \tau(X_{j_{t_k+1}} \dots X_{j_\ell}).
		\end{equation}
		By grouping the terms based on the value of $i_1 = j_{t_1}$, \dots, $i_k = j_{t_k}$, this is equal to
		\begin{equation}
		\sum_{\vec{\imath} \in [m]^k} K_{k+1}(X_i, X_{i_1},\dots,X_{i_k}) \tau^{\otimes (k+1)}[\partial_{\vec{\imath}} f(X)].
		\end{equation}
		Finally, we sum over $k$ which produces the right-hand side of \eqref{eq: Stein identity}.
	\end{proof}
	
	The Stein identity leads to a formula expressing $X_i$ as a series in terms of the higher-order conjugate variables, provided the series converges.
	
	\begin{proposition} \label{prop: conjugate variable identity}
		Let $X = (X_1,\dots,X_m)$ be a tuple of noncommutative random variables that generate $(M,\tau)$.  Suppose that for every $k \in \N$ and $\vec{\imath}$, $X$ admits an $\vec{\imath}$-conjugate variable $\xi_{\vec{\imath}}$.  Moreover, suppose that
		\begin{equation}
		\lim_{\ell \to \infty} \sum_{k=0}^\ell \sum_{\vec{\imath} \in [m]^k} \overline{K_{k+1}(X_i,X_{i_1},\dots,X_{i_k})} \xi_{\vec{\imath}} \text{ exists in } L^2(M,\tau) \text{ for } i \in [m].
		\end{equation}
		Then we have
		\begin{equation} \label{eq: conjugate variable identity}
			X_i = \lim_{\ell \to \infty} \sum_{k=0}^\ell \sum_{\vec{\imath} \in [m]^k} \overline{K_{k+1}(X_i,X_{i_1},\dots,X_{i_k})} \xi_{\vec{\imath}} \text{ for } i \in [m].
		\end{equation}
	\end{proposition}
	
	\begin{proof}
		Let $\eta_i$ be the right-hand side of \eqref{eq: conjugate variable identity}.  For each polynomial $f$, if $\ell \geq \deg(f)$, we have by \eqref{eq: Stein identity} and \eqref{eq: higher order conjugate} that
		\begin{align}
			\tau(X_i f(X)) &= \sum_{k=0}^\ell \sum_{\vec{\imath} \in [m]^k} K_{k+1}(X_i,X_{i_1},\dots,X_{i_k}) \tau^{\otimes (k+1)}[\partial_{\vec{\imath}} f(X)] \\
			&= \sum_{k=0}^\ell \sum_{\vec{\imath} \in [m]^k} K_{k+1}(X_i,X_{i_1},\dots,X_{i_k}) \ip{\xi_{\vec{\imath}} f(X)}_\tau.
		\end{align}
		Taking $\ell \to \infty$, we have $\tau(X_i f(X)) = \tau(\eta_i f(X))$ for each $i \in [m]$, hence $X_i = \eta_i$ since polynomials are dense in $L^2(M,\tau)$.
	\end{proof}
	
	We will apply this identity to the linear combination of a given tuple $X$ with a free semicircular family $S$, in conjunction with the behavior of higher-order conjugate variables under conditioning \cite[Proposition 3.7]{VoiculescuFE5}.
	
	\begin{proposition} \label{prop: main conjugate variable bound}
		Let $X$ be a self-adjoint $m$-tuple and let $S$ be a freely independent standard semicircular $m$-tuple.  Let $a, b \in \R \setminus \{0\}$ with $a^2 + b^2 = 1$, and let $Y = aX + bS$.  Then $Y$ has conjugate variables of all orders given by
		\begin{equation} \label{eq: conditional expectation higher order}
			\xi_{\vec{\imath}} = b^{-k} E_{\mathrm{W}^*(Y)}[U_{\vec{\imath}}(S)].
		\end{equation}
		Furthermore, we have
		\begin{equation} \label{eq: main conjugate variable bound}
			\norm{Y_i - \xi_i}_2^2 \leq \sum_{k=0}^\infty \sum_{\vec{\imath} \in [m]^k} a^{2k+2} b^{-2k} |K_{k+1}(X_i,X_{i_1},\dots,X_{i_k}) - K_{k+1}(S_i,S_{i_1},\dots,S_{i_k})|^2,
		\end{equation}
		Here note that $K_{k+1}(S_1,S_{i_1},\dots,S_{i_k}) = \delta_{k=1} \delta_{i,i_1}$.
	\end{proposition}
	
	\begin{proof}
		Consider the $2m$-tuple $(S,X)$.  For $\vec{\imath} \in [m]^k$, the variable $U_{\vec{\imath}}(S)$ is a $\vec{\imath}$-conjugate variable for $(S,X)$ by the same argument as \cite[Proposition 3.6]{VoiculescuFE5}.  Then as in \cite[Proposition 3.7]{VoiculescuFE5},
		\begin{align}
			\ip{U_{\vec{\imath}}(S), f(Y)} &= \ip{U_{\vec{\imath}}(S), f(aX + bS)} \\
			&= \tau^{\otimes (k+1)}[\partial_{i_1,\dots,i_k}^S f(aX + bS)] \\
			&= b^k \tau^{\otimes (k+1)}[\partial_{i_1,\dots,i_k}^Y f(Y)],
		\end{align}
		where the superscripts $S$ and $Y$ indicate the variables of differentiation.  This yields \eqref{eq: conditional expectation higher order}.
		
		To prove \eqref{eq: main conjugate variable bound}, assume without loss of generality that the right-hand side is finite.  Moreover, note that for $\vec{\imath} \in [m]^k$,
		\begin{equation}
		K_{k+1}(Y_i,Y_{i_1},\dots,Y_{i_k}) = a^{k+1} K_{k+1}(X_i,X_{i_1},\dots,X_{i_k}) + b^{k+1} K_{k+1}(S_i,S_{i_1},\dots,S_{i_k})
		\end{equation}
		by vanishing of mixed cumulants under freeness \cite[Theorem 11.15]{NS2006}.  In particular,
		\begin{equation}
		\sum_{k=0}^\infty \sum_{\vec{\imath} \in [m]^k} |b^{-k} K_{k+1}(Y_i,Y_{i_1},\dots,Y_{i_k})|^2 < \infty.
		\end{equation}
		Since the Chebyshev polynomials are orthonormal for the semicircular family, the following is a convergent orthogonal sum in $L^2(M,\tau)$:
		\begin{equation}
		\eta_i = \sum_{k=0}^\infty \sum_{\vec{\imath} \in [m]^k} b^{-k} \overline{K(Y_i,Y_{i_1},\dots,Y_{i_k})} U_{\vec{\imath}}(S).
		\end{equation}
		Taking the conditional expectation onto $\mathrm{W}^*(Y)$, we obtain
		\begin{equation}
		E_{\mathrm{W}^*(Y)}[\eta_i] = \lim_{\ell \to \infty} \sum_{k=0}^\ell \sum_{\vec{\imath} \in [m]^k} b^{-k} \overline{K(Y_i,Y_{i_1},\dots,Y_{i_k})} \xi_{\vec{\imath}} = Y_i
		\end{equation}
		by Proposition \ref{prop: conjugate variable identity}.  Thus,
		\begin{equation}
		Y_i = E_{\mathrm{W}*(Y)} \left[ \sum_{k=0}^\infty \sum_{\vec{\imath} \in [m]^k} \left( a^{k+1} b^{-k} \overline{K(X_i,X_{i_1},\dots,X_{i_k})} + b \overline{K(S_i,S_{i_1},\dots,S_{i_k})} \right) U_{\vec{\imath}}(S) \right].
		\end{equation}
		Subtract $\xi_i = b^{-1} E_{\mathrm{W}^*(Y)}[U_i(S)]$ from both sides to obtain
		\begin{equation}
		Y_i - \xi_i = E_{\mathrm{W}*(Y)} \left[ \sum_{k=0}^\infty \sum_{\vec{\imath} \in [m]^k} a^{k+1} b^{-k} \left( \overline{K(X_i,X_{i_1},\dots,X_{i_k})} - \overline{K(S_i,S_{i_1},\dots,S_{i_k})} \right) U_{\vec{\imath}}(S) \right],
		\end{equation}
		noting that $(b^{-1} - b) = b^{-1}(1 - b^2) = b^{-1} a^2$ and only one of the cumulants of the semicirculars is nonzero.  Thus, we conclude \eqref{eq: main conjugate variable bound} using the contractivity of the conditional expectation and orthonormality of $U_{\vec{\imath}}(S)$.
	\end{proof}
	
	\subsection{Upper bound for main example}
	
	To apply Proposition \ref{prop: main conjugate variable bound} to our counterexample in the main theorem, we need to estimate the free cumulants of $(X_1,\dots,X_{m+1})$.
	
	\begin{proposition} \label{prop: graph independent bound}
		Let $(X_1,\dots,X_m)$ be a standard free semicircular family and let $X_{m+1}$ be a semicircular operator in tensor position with $(X_1,\dots,X_m)$.  For $\vec{\imath} = (i_1,\dots,i_\ell) \in \{0,\dots,m\}^\ell$, let $z(\vec{\imath})$ be the number of occurrences of $0$ in $\vec{\imath}$. Then the free cumulants satisfy
		\begin{align}
			K_1(X_i) &= 0, \\
			K_2(X_{i_1},X_{i_2}) &= \delta_{i_1,i_2} \\
			K_\ell(X_{i_1},\dots,X_{i_\ell}) &= 0 \text{ for odd } \ell \\
			0 \leq K_{\ell}(X_{i_1},\dots,X_{i_{\ell}}) &\leq 2^{\ell} \text{ for even } \ell. \label{eq: cumulant bound for even ell}
		\end{align}
		Moreover,
		\begin{equation} \label{eq: L1 sum of cumulants}
		\sum_{\vec{\imath} \in [m+1]^{\ell}} K_\ell(X_{i_1},\dots,X_{i_{\ell}}) \leq 2^{\ell}(1 + m^{1/2})^{\ell} \text{ for even } \ell.
		\end{equation}
		and
		\begin{equation} \label{eq: L2 sum of cumulants}
			\sum_{\vec{\imath} \in [m+1]^{\ell}} K_\ell(X_{i_1},\dots,X_{i_{\ell}})^2 \leq 4^\ell (1 + m^{1/2})^{\ell} \text{ for even } \ell.
		\end{equation}
	\end{proposition}
	
	\begin{proof}
		The first and second order cumulants are simply the mean and covariance, so the first two lines are immediate and we focus on the cumulants of order at least three.  Tensor independence is a special case of graph independence (also known as $\varepsilon$-independence or $\Lambda$-independence).  This allows us to apply a formula of Speicher and Wysocza{\'n}ski \cite{SpWys2016} for free cumulants in a graph product, here in the special case of a tensor product.
		
		Consider a partition of $[\ell]$ into two sets $A$ and $B$ and a partition $\pi$ of $[\ell]$.  We say that $\pi \in \NC(A,B)$ if $\pi \leq \{A,B\}$ and $\pi|_A$ and $\pi|_B$ are non-crossing.  We say that $\pi \leq \pi'$ if every block of $\pi$ is contained in a block of $\pi'$.  Moreover, the \emph{non-crossing envelope} of $\pi$ is smallest non-crossing partition $\pi'$ satisfying $\pi \leq \pi'$.  Finally, let $\NC^0(A,B)$ be the set of $\pi \in \NC(A,B)$ such that the non-crossing envelope of $\pi$ is the partition where all the elements are in one block.  Then Speicher and Wysocza{\'n}ski's formula \cite[Theorem 5.2]{SpWys2016} tells us that if $X_{m+1}$ is in tensor position with $(X_1,\dots,X_m)$.
		\begin{equation} \label{eq: partition expansion}
		K_\ell(X_{i_1} \dots X_{i_\ell}) = \sum_{\pi \in NC^0(A(\vec{\imath}),B(\vec{\imath}))} K_\pi(X_{i_1},\dots,X_{i_\ell}),
		\end{equation}
		where $A(\vec{\imath})$ is the set of indices where $i_j = 0$ and $B(\vec{\imath})$ is the set of indices where $i_j > 0$.  Now because $X_{m+1}$ and $(X_1,\dots,X_m)$ are semicircular, $K_\pi$ will vanish unless $\pi$ is a pair partition.  In particular, $K_\ell(X_{i_1},\dots,X_{i_\ell})$ vanishes when $\ell$ is odd.  In the even case, the only $\pi$ that contribute are pair partitions such that $\pi|_A$ and $\pi|_B$ are non-crossing.  The numbers of non-crossing pair partitions of $A$ and $B$ respectively are bounded by $2^{|A|}$ and $2^{|B|}$ respectively, which follows from the fact that the number of pair partitions of $[\ell]$ is the $\ell$th moment of the standard semicircular distribution (see Proposition \ref{prop: cumulants of semicircular}).  Therefore,
		\begin{equation}
		K_\pi(X_{i_1},\dots,X_{i_\ell}) \leq 2^{|A|} 2^{|B|} = 2^\ell.
		\end{equation}
		
		To prove \eqref{eq: L1 sum of cumulants}, note that by free independence, the contribution of a partition $\pi$ is zero unless $\pi$ respects the coloring $\vec{\imath}$, meaning that $s$ and $t$ are in the same block of $\pi$, then $i_s = i_t$.  For a fixed $A$ and $B$, we write
		\begin{equation}
		\sum_{\substack{\vec{\imath} \in [m+1]^\ell: \\ A(\vec{\imath}) = A}}
		K_\ell(X_{i_1} \dots X_{i_\ell}) = \sum_{\pi \in NC^0(A,B)} \sum_{\substack{\vec{\imath} \in [m+1]^\ell: \\ A(\vec{\imath}) = A}} K_\pi(X_{i_1},\dots,X_{i_\ell}).
		\end{equation}
		If we fix $\pi$ appearing in the sum, the number of pairs in $A$ is $|A|/2$ and the number of pairs in $B$ is $|B|/2$.  Each pair in $A$ has the variable $X_{m+1}$ and each pair in $B$ must be assigned an index in $[m]$.  This means there are $m^{|B|/2}$ choices of indices, so $\pi$ contributes $m^{|B|/2}$ to the sum.  Moreover, as noted above, the number of pair partitions $\pi \in \NC^0(A,B)$ is at most $2^\ell$. Therefore,
		\begin{equation}
		\sum_{\substack{\vec{\imath} \in [m+1]^\ell: \\ A(\vec{\imath}) = A}} K_\pi(X_{i_1},\dots,X_{i_\ell}) \leq 2^\ell m^{|B|/2}.
		\end{equation}
		Finally, we sum over the choices of $(A,B)$ using the binomial expansion to get $2^\ell (1 + m^{1/2})^\ell$, which finishes the proof of \eqref{eq: L1 sum of cumulants}.  Next, \eqref{eq: L2 sum of cumulants} follows by combining \eqref{eq: cumulant bound for even ell} and \eqref{eq: L1 sum of cumulants}, since for any vector $v$, $\norm{v}_{\ell^2} \leq \norm{v}_{\ell^1}^{1/2} \norm{v}_{\ell^\infty}^{1/2}$.
	\end{proof}
	
	\begin{proposition}
		Let $X = (X_1,\dots,X_{m+1})$ be as in Proposition \ref{prop: graph independent bound}.  Let $S = (S_1,\dots,S_{m+1})$ be a standard semicircular family freely independent from $X$.  Let $a, b \in \R \setminus \{0\}$ with $a^2 + b^2 = 1$.  Let $Y = aX + bS$ and let $\xi_i$ be the $i$th conjugate variable for $Y$.  Then whenever $4a^2 (1 + m^{1/2}) < b^2$, we have
		\begin{equation} \label{eq: example free Fisher upper bound}
			\sum_{i=0}^m \norm{\xi_i - Y_i}_2^2 \leq b^{-2} \left[ \frac{(2a)^8 (m^{1/2} + 1)^4}{b^4 - (2a)^4(m^{1/2} + 1)^2} \right].
		\end{equation} 
	\end{proposition}
	
	\begin{proof}
		We apply \eqref{eq: main conjugate variable bound} for $i$.  First, note that the $k = 0$ and $k = 1$ terms vanish since $X$ has mean zero and covariance identity.  Hence, we have
		\begin{align}
			\sum_{i=0}^m \norm{Y_i - \xi_i}_2^2 &\leq \sum_{k=2}^\infty \sum_{\vec{\imath} \in [m+1]^k} a^{2k+2} b^{-2k} |K_{k+1}(X_i,X_{i_1},\dots,X_{i_k})|^2 \\
			&\leq b^2 \sum_{k=3}^\infty \sum_{\vec{\imath} \in [m+1]^k} a^{2k} b^{-2k} |K_k(X_{i_1},\dots,X_{i_k})|^2.
		\end{align}
		By Proposition \ref{prop: graph independent bound}, the odd terms are zero and for even $k$,
		\begin{equation}
		\sum_{\vec{\imath} \in [m+1]^k} |K_k(X_{i_1},\dots,X_{i_k})|^2 \leq 2^{2k} (m^{1/2}+1)^k.
		\end{equation}
		Thus, taking $k = 2\ell$,
		\begin{equation}
		\sum_{k=3}^\infty \sum_{\vec{\imath} \in [m+1]^k} a^{2k} b^{-2k} |K_k(X_{i_1},\dots,X_{i_k})|^2 \leq \sum_{\ell=2}^\infty a^{4\ell} b^{-4\ell} 2^{4 \ell} (m^{1/2}+1)^{2\ell} = \frac{[2a/b]^8(m^{1/2}+1)^4}{1 - [2a/b]^4(m^{1/2}+1)^2}.
		\end{equation}
		The asserted formula then follows by algebraic manipulations.
	\end{proof}

	\section{Detection of block structure} \label{sec: block detection}
	
	We will use the information-cost inequality to lower-bound the square root of the Fisher information of $Y^{(n)}(t)$ by the Wasserstein distance $W_2(Y^{(n)}(t),S^{(n)})$, and in turn bound this below by $\mathbb{E} [\psi_k^{(n)}(Y^{(n)}(t))] - \mathbb{E}[\psi(S^{(n)})]$ for some Lipschitz function $\psi$.  The function $\psi_k^{(n)}$ measures the maximum $2$-norm of $\sum_j p_jxp_j$ where $p_1$, \dots, $p_k$ are projections of trace $\approx 1/k$.  This can be understood as a way of detecting block structure as in OpenAI's argument.
	
	We will show that $\mathbb{E}[\psi_k^{(n)}(S^{(n)})]$ is relatively small, while $\mathbb{E}[\psi_k^{(n)}(Y^{(n)}(t))]$ is comparatively larger for the right choices of parameters.  To motivate the lower bound for $Y^{(n)}(t)$, note that since $X_{m+1}$ commutes with $X_1$, \dots, $X_m$, we can naturally use spectral projections of $X_{m+1}$ to produce a block structure for $X$, namely,
	\begin{equation}
	X = \sum_{j=1}^k p_j X p_j, \text{ where } p_j = \mathbbm{1}_{[\lambda_{j-1},\lambda_j]}(X_{m+1})
	\end{equation}
	for appropriate $\lambda_0 < \dots < \lambda_k$.  On the other hand, the free semicircular family $S$ fails to commute with the $p_j$'s, and thus is as far as possible from having such a block decomposition.  Hence, $\sum_{j=1}^k p_j X(t) p_j$ will be approximately $e^{-t/2} X(0)$.  We therefore expect $\mathbb{E}[\psi_k^{(n)}(Y^{(n)}(t))]$ to be approximately $\sqrt{m+1} e^{-(t_0+t)/2}$.
	
	\begin{notation} \label{not: partition test}
		Fix $n$ and $k$.  Let $q$ and $r$ be the integer quotient and remainder upon dividing $n$ by $k$, so $n = kq + r$.  Let $\mathcal{P}_k^{(n)}$ be the collection of $k$-tuples of projections $(p_1,\dots,p_k)$ in $\mathbb{M}_n$ such that
		\begin{equation}
			\sum_{j=1}^k p_j = 1, \qquad
			\tr_n(p_j) = \frac{q+1}{n} \text{ for } j = 1, \dots r, \qquad
			\tr_n(p_j) = \frac{q}{n} \text{ for } j = r+1, \dots, k.
		\end{equation}
		For $p \in \mathcal{P}_k^{(n)}$, define
		\begin{equation} \label{eq: paving}
			\mathcal{E}_p(x) = \sum_{j=1}^k p_jxp_j, \text{ where } (p_jx_1p_j,\dots,p_jx_mp_j).
		\end{equation}
		Then define
		\begin{equation} \label{eq: def of psi}
			\psi_k^{(n)}(x) = \sup_{p \in \mathcal{P}_k^{(n)}} \norm*{\mathcal{E}_p(x)}_2.
		\end{equation}
	\end{notation}
	
	\subsection{Upper bound for Gaussian case}
	
	The first part of the argument will be to obtain an upper bound for $\mathbb{E}[\psi_k^{(n)}(S^{(n)})]$.  Since we work only with the Gaussian in this section, we return to the setting of $m$ variables, that is, $S^{(n)} = (S_1^{(n)},\dots,S_m^{(n)}) \sim \gamma_m^{(n)}$, and then substitute $m+1$ for $m$ later when analyzing the main example.  We require a uniform upper bound for $\norm*{\sum_{j=1}^k p_j x p_j}_2$ to hold for all $p \in \mathcal{P}_k^{(n)}$ with high probability.   We first bound the expectation for each $p$, and afterward we derive a uniform bound using concentration inequalities.
	
	\begin{lemma}
		For $p \in \mathcal{P}_k^{(n)}$,
		\begin{equation} \label{eq: upper bound on expectation}
			\mathbb{E} \norm*{\mathcal{E}_p(S^{(n)})}_2 \leq \sqrt{m} \left( \frac{1}{k} + \frac{1}{n} \right)^{1/2}.
		\end{equation}
	\end{lemma}
	
	\begin{proof}
		By the Cauchy-Schwarz inequality,
		\begin{align}
			\left( \mathbb{E} \norm*{\sum_{j=1}^k p_j S^{(n)} p_j}_2 \right)^2 &\leq \mathbb{E} \left[ \norm*{\sum_{j=1}^k p_j S^{(n)} p_j}_2^2 \right] \\
			&= \sum_{i=1}^m \sum_{j=1}^k \mathbb{E} \tr_n[p_j S_i^{(n)} p_j S_i^{(n)} p_j] = \sum_{i=1}^m \sum_{j=1}^k \tr_n(p_j)^2 \\
			&\leq m \frac{q+1}{n} \sum_{j=1}^k \tr_n(p_j) \leq m \left( \frac{1}{k} + \frac{1}{n} \right),
		\end{align}
		and taking square roots completes the proof.
	\end{proof}
	
	To aid in the concentration computation, we first note the following facts about $\psi_k^{(n)}$.
	
	\begin{observation}
		Since all elements in $\mathcal{P}_k^{(n)}$ are conjugate by a unitary, we have for each fixed $p \in \mathcal{P}_k^{(n)}$
		\begin{equation} \label{eq: expressing psi via unitary orbit}
			\psi_k^{(n)}(x) = \inf_{u \in \mathbb{U}_n} \norm{\mathcal{E}_{upu^*}(x)}_2 = \inf_{u \in \mathbb{U}_n} \norm{\mathcal{E}_{p}(u^*xu)}_2,
		\end{equation}
		where $\mathcal{U}_n$ is the $n \times n$ unitary group.  For $x, y \in (\mathbb{M}_n)_{\sa}^m$ and $p \in \mathcal{P}_k^{(n)}$,
		\begin{equation} \label{eq: psi Lipschitz estimate 1}
			\norm{\mathcal{E}_p(x) - \mathbb{E}_p(y)}_2 \leq \norm{x - y}_2.
		\end{equation}
		Hence for unitaries $u, v \in \mathbb{U}_n$ and $x \in (\mathbb{M}_n)_{\sa}^m$,
		\begin{equation} \label{eq: psi Lipschitz estimate 2}
			\norm{\mathcal{E}_p(uxu^*) - \mathbb{E}_{p}(vxv^*)}_2 \leq \norm{uxu^* - vxv^*}_2 \leq 2 \norm{u - v}_\infty \norm{x}_2.
		\end{equation}
		Moreover, \eqref{eq: psi Lipschitz estimate 1} implies that $\psi$ is the infimum of a family of $1$-Lipschitz functions, so
		\begin{equation} \label{eq: psi Lipschitz estimate 3}
			|\psi_k^{(n)}(x) - \psi_k^{(n)}(y)| \leq \norm{x - y}_2.
		\end{equation}
	\end{observation}
	
	\begin{lemma} \label{lem: epsilon net}
		For $m \geq 2$, we have
		\begin{equation} \label{eq: final expectation estimate}
			\mathbb{E} \psi(S^{(n)}) \leq \sqrt{m} \left( \frac{1}{k} + \frac{1}{n} \right)^{1/2} + 1 + \sqrt{2 + \log m}
		\end{equation}
	\end{lemma}
	
	\begin{proof}
		Fix $p \in \mathcal{P}_k^{(n)}$.  Let $\delta > 0$.  Let $\mathcal{Z}$ be a maximal subset $\mathbb{U}_n$ that are $\delta/2$-separated with respect to $\norm{\cdot}_\infty$, that is, $\norm{u - v}_\infty \geq \delta / 2$ for $u \neq v$ in $\mathcal{Z}$.  Since the $\delta / 2$-balls in $\mathbb{M}_n$ centered at $u \in \mathcal{Z}$ are disjoint and contained in the ball of radius $1 + \delta / 2$ centered at the origin, and since the Lebesgue measure of a $\norm{\cdot}_\infty$-ball in $\mathbb{M}_n$ scales by the radius to the power $2n^2$, we have $|\mathcal{Z}| \leq ((2 + \delta) / \delta)^{2n^2}$.  Moreover, the $\norm{\cdot}_\infty$-balls of radius $\delta$ centered at $z \in \mathcal{Z}$ cover $\mathbb{U}_n$.  Let
		\begin{equation}
		\psi_{\mathcal{Z}}(x) = \max_{u \in \mathcal{Z}} \norm{\mathcal{E}_p(uxu^*)}_2
		\end{equation}
		Using \eqref{eq: expressing psi via unitary orbit}, \eqref{eq: psi Lipschitz estimate 2}, and the fact $\mathcal{Z}$ is a $\delta$-covering of $\mathbb{U}_n$,
		\begin{equation}
		\psi_k^{(n)}(S^{(n)}) \leq \psi_{\mathcal{Z}}(S^{(n)}) + 2 \delta \norm{S^{(n)}}_2.
		\end{equation}
		Moreover, $\mathbb{E} \norm{S^{(n)}}_2 \leq \mathbb{E}[\norm{S^{(n)}}_2^2]^{1/2} = \sqrt{m}$, and hence
		\begin{equation} \label{eq: delta net approximation}
			\mathbb{E} \psi_k^{(n)}(S^{(n)}) \leq \mathbb{E} \psi_{\mathcal{Z}}(S^{(n)}) + 2 \delta \sqrt{m}.
		\end{equation}
		It remains to estimate $\mathbb{E} \psi_{\mathcal{Z}}(S^{(n)})$, for which we will use an exponential moment bound.  We recall the \emph{Herbst concentration inequality} (see \cite[Lemmas 2.3.2 and 2.3.3]{AGZ2009}).  With the normalization of $\norm{\cdot}_2$ and $\gamma_m^{(n)}$, the log-Sobolev constant of $S^{(n)}$ is $1/n^2$, so this inequality says that for Lipschitz $f: (\mathbb{M}_n)_{\sa}^m \to \R$ and $\lambda \in \R$,
		\begin{equation} \label{eq: Herbst exponential moment}
			\mathbb{E} \exp(\lambda [f(S^{(n)}) - \mathbb{E} f(S^{(n)})]) \leq \exp(\lambda^2 n^{-2} \norm{f}_{\Lip}^2 / 2)
		\end{equation}
		Now substitute $-n^2 \lambda$ for $\lambda$, take $f(S^{(n)}) = \norm{\mathcal{E}_p(uS^{(n)}u^*)}_2$, note $f$ is $1$-Lipschitz by \eqref{eq: psi Lipschitz estimate 1}, and bound $\mathbb{E} f(S^{(n)})$ by \eqref{eq: upper bound on expectation} to obtain
		\begin{equation}
		\mathbb{E} \exp \left( -n^2 \lambda \left( \norm{\mathcal{E}_p(uS^{(n)}u)}_2 - \sqrt{m}\left(\frac{1}{k} + \frac{1}{n} \right)^{1/2} \right) \right) \leq \exp(\lambda^2 n^2 / 2).
		\end{equation}
		Therefore, we have for $\lambda > 0$,
		\begin{align}
			\mathbb{E} \exp \left( n^2 \lambda \left( \psi_{\mathcal{Z}}(S^{(n)}) - \sqrt{m}\left(\frac{1}{k} + \frac{1}{n} \right)^{1/2} \right) \right) &= \mathbb{E} \max_{u \in \mathcal{Z}} \exp \left( -n^2 \lambda \left( \norm{\mathcal{E}_p(uS^{(n)}u)}_2 - \sqrt{m}\left(\frac{1}{k} + \frac{1}{n} \right)^{1/2} \right) \right) \\
			&\leq \sum_{u \in \mathcal{Z}} \mathbb{E} \exp \left( -n^2 \lambda \left( \norm{\mathcal{E}_p(uS^{(n)}u)}_2 - \sqrt{m}\left(\frac{1}{k} + \frac{1}{n} \right)^{1/2} \right) \right) \\
			&\leq \left(\frac{2 + \delta}{\delta} \right)^{2n^2} \exp(\lambda^2 n^2 / 2).
		\end{align}
		Applying Jensen's inequality,
		\begin{equation}
		\exp \left( n^2 \lambda \left( \mathbb{E}  \psi_{\mathcal{Z}}(S^{(n)}) - \sqrt{m}\left(\frac{1}{k} + \frac{1}{n} \right) \right) \right) \leq \left(\frac{2 + \delta}{\delta} \right)^{2n^2} \exp(\lambda^2 n^2 / 2),
		\end{equation}
		and hence
		\begin{equation}
		\mathbb{E} \psi_{\mathcal{Z}}(S^{(n)}) \leq \sqrt{m}\left(\frac{1}{k} + \frac{1}{n} \right)^{1/2} + \frac{\lambda}{2} + \lambda^{-1} \log \frac{2 + \delta}{\delta}.
		\end{equation}
		Plugging in the optimal choice $\lambda = \sqrt{2 \log \frac{2 + \delta}{\delta}}$ and applying \eqref{eq: delta net approximation},
		\begin{equation} \label{eq: penultimate expectation estimate}
			\mathbb{E} \psi(S^{(n)}) \leq \sqrt{m}\left(\frac{1}{k} + \frac{1}{n} \right)^{1/2} + 2 \delta \sqrt{m} + \sqrt{2 \log \frac{2 + \delta}{\delta}}.
		\end{equation}
		In particular, taking $\delta = m^{-1/2}$ and noting $2 + m^{-1/2} \leq 2 + 2^{-1/2} \leq e$, we have
		\begin{equation}
		\delta \sqrt{m} + \sqrt{2 \log \frac{2 + \delta}{\delta}} \leq 1 + \sqrt{2 \log \frac{e}{m^{-1/2}}} = 1 + \sqrt{2 + \log m}.
		\end{equation}
		Substituting this into \eqref{eq: penultimate expectation estimate} yields \eqref{eq: final expectation estimate}.
	\end{proof}
	
	\subsection{Lower bound for main example}
	
	Assumption \ref{ass: main example setup} guarantees that $X_{m+1}$ commutes with $X_1$, \dots, $X_m$ (and also with itself) and therefore any spectral projections $p_j$ of $X_{m+1}$ will satisfy $\sum_{j=1}^k p_j X p_j = X$.  However, we have to consider random matrix models for the perturbed version $X(t_0) \sim Y(0)$, whose von Neumann algebra does not \emph{a priori} contain the spectral projections of $p_j$.  To deal with the perturbation (as well as the matrix approximation in the large-$n$ limit), it is convenient to replace the projection with something that depends continuously on the input operator, i.e.\ replace the indicator function corresponding to each spectral projection with a continuous cutoff function.  This natural approximation was carried out by OpenAI at the beginning of \cite[\S 6]{OpenAI2026entropy}.  The next lemma gives an explicit control of the $L^1$ error when the measure $\mu$ has bounded density.
	
	\begin{observation} \label{obs: quantiles}
		Let $\mu$ be a probability measure on $[-R,R]$ with bounded density $\rho \leq C$.  Let $\lambda_0 = -R$, $\lambda_1$, \dots, $\lambda_k = R$ be a sequence of $k$-quantiles for $\mu$, that is,
		\begin{equation}
			\mu([-R,\lambda_j]) = \frac{j}{k} \text{ for } j = 0, \dots, k.
		\end{equation}
		Observe that
		\begin{equation}
			\lambda_j - \lambda_{j-1} \geq \frac{1}{C} \mu([\lambda_{j-1},\lambda_j]) = \frac{1}{Ck}.
		\end{equation}
		Fix $L \geq 1$, and let $f_j: [-R,R] \to [0,1]$ be given by
		\begin{equation} \label{eq: intervals and functions}
			f_{L,j}(t) = \min(1, L d(t,\R \setminus I_j)), \qquad \text{ where } I_j =  \begin{cases}
				(-\infty,\lambda_1], & j = 1, \\
				(\lambda_{j-1},\lambda_j], & j = 2, \dots, k-1 \\
				(\lambda_{k-1},\infty), & j = k,
			\end{cases}
		\end{equation}
		which is $L$-Lipschitz.  Thus, $f_j$ is supported in $[\lambda_{j-1},\lambda_j]$, and one moves inward from the edge of the interval, $f$ increases with slope $L$, but cuts off if it reaches $1$.  Hence, for $j = 2, \dots, k-1$,
		\begin{equation}
		\int_{[\lambda_{j-1},\lambda_j]} f_{L,j}^2 \,d\mu \geq \mu([\lambda_{j-1},\lambda_j]) - \mu([\lambda_{j-1},\lambda_{j-1}+1/L]) - \mu([\lambda_j-1/L,\lambda_j]) \geq \frac{1}{k} - \frac{2C}{L},
		\end{equation}
		so that
		\begin{equation} \label{eq: function integral deficit}
			\frac{1}{k} - \frac{2C}{L} \leq \int_{[\lambda_{j-1},\lambda_j]} f_{L,j}^2\,d\mu \leq  \int_{[\lambda_{j-1},\lambda_j]} f_{L,j}\,d\mu \leq \frac{1}{k}.
		\end{equation}
		The functions $f_{L,j}$ also have disjoint support.
	\end{observation}
	
	The next lemma allows us to build projections with specified traces as close as possible to given smoothed-out versions.  We remark that it is well known that any operator that satisfies the relations of a projection approximately must be close to a projection (and similarly for matrix units), which has had many applications in operator algebras.
	
	\begin{lemma} \label{lem: building approximate projections}
		Let $\lambda_j$ and $f_{L,j}$ be as in the previous observation.  Let $y \in (\mathbb{M}_n)_{\sa}$.  Write $n = qk + r$ by integer division.  Then there exists projections $(p_j^{\star})_{j=1}^k$ satisfying
		\begin{equation} \label{eq: p star properties}
			\sum_{j=1}^k p_j^{\star} = 1, \qquad \tr_n(p_j^{\star}) = \frac{q + \mathbbm{1}_{j \leq r}}{n}
		\end{equation}
		and
		\begin{equation} \label{eq: p star error}
			\sum_{j=1}^k \norm{p_j^{\star} - f_{L,j}(y)}_1 \leq 3 \sum_{j=1}^k \left| \frac{1}{k} - \tr_n(f_{L,j}(y)) \right| + \frac{k}{n}.
		\end{equation}
	\end{lemma}
	
	\begin{proof}
		For the first claim, recall $y$ is contained in some maximal abelian subalgebra $A$, which is conjugate to the diagonal subalgebra.   Let $I_j$ be as in \eqref{eq: intervals and functions}, and consider the projections
		\begin{equation}
		p_j = \mathbbm{1}_{I_j}(y).
		\end{equation}
		Let $p_j'$ be a projection in $A$ such that
		\begin{equation}
		\tr_n(p_j') = \min\left( \frac{q+\mathbbm{1}_{j\leq r}}{n}, \tr_n(p_j) \right),
		\end{equation}
		and let $J$ be the set of indices where $\tr_n(p_j) > (q+\mathbbm{1}_{j \leq r})/n$.  Let $p_0' = 1 - \sum_{j=1}^k p_j'$.  We can then redistribute the mass of $p_0'$ to the projections $p_j'$ for $j \in [k] \setminus J$ whose trace is too small.  Namely, there exists projections $p_j'' \in A$ satisfying
		\begin{equation}
		p_j'' = 0 \text{ for } j \in J, \qquad \sum_{j=1}^k p_j'' = p_0', \qquad \tr_n(p_j'' + p_j') = \frac{q + \mathbbm{1}_{j\leq r}}{n}.
		\end{equation}
		Let $p_j^{\star} = p_j' + p_j''$.  Note that
		\begin{equation}
		\sum_{j=1}^k \norm{p_j^{\star} - p_j}_1 \leq \sum_{j=1}^k \tr_n(p_j^{\star} - p_j') + \sum_{j=1}^k \tr_n(p_j - p_j') = 2 \sum_{j=1}^k \tr_n(p_j - p_j'),
		\end{equation}
		since $\sum_j \tr_n(p_j) = 1 = \sum_j \tr_n(p_j^{\star})$.  By construction, for $j \in J$,
		\begin{equation} \label{eq: estimate on difference of projections}
		\tr_n(p_j - p_j') = \tr_n(p_j) - \frac{q+\mathbbm{1}_{j\leq r}}{n} \leq \tr_n(p_j) - \frac{1}{k} + \frac{1}{n} \leq \tr_n(p_j - f_{L,j}(y)) + \left|\tr_n(f_{L,j}(y)) - \frac{1}{k} \right| + \frac{1}{n},
		\end{equation}
		where we note that $p_j - f_{L,j}(y) = \mathbbm{1}_{I_j}(y) - f_{L,j}(y)$.  For $j \not \in J$, the left-hand side of \eqref{eq: estimate on difference of projections} is zero.  Now
		\begin{equation}
		\sum_{j=1}^k \tr_n(p_j - f_{L,j}(y)) = 1 - \sum_{j=1}^k \tr_n(f_{L,j}(y)) \leq \sum_{j=1}^k \left| \frac{1}{k} - \tr_n(f_{L,j}(y)) \right|.
		\end{equation}
		Thus,
		\begin{equation}
		\sum_{j=1}^k \norm{p_j^{\star} - f_{L,j}(y)}_1 \leq \sum_{j=1}^k \norm{p_j - f_{L,j}(y)}_1 + \sum_{j=1}^k \norm{p_j - p_j^{\star}}_1 \leq 3 \sum_{j=1}^k \left| \frac{1}{k} - \tr_n(f_{L,j}(y)) \right| + \frac{k}{n}.
		\end{equation}
		This finishes the proof of the first claim, including \eqref{eq: p star properties} and \eqref{eq: p star error}.
	\end{proof}
	
	Next, we need to estimate the error in the paving function resulting from replacing $p_j^{\star}$ with $f_{L,j}(y)$, which will use the following general computation.
	
	\begin{lemma}
		For a family of positive operators $a = (a_1,\dots,a_k)$ in a tracial von Neumann algebra with $\sum_{j=1}^k a_j \leq 1$, define $\mathcal{E}_a(x) = \sum_{j=1}^k a_j x a_j$.  Then for any two such families $a$ and $b$, and $x \in M^m$,
		\begin{equation} \label{eq: compression error}
			\norm{\mathcal{E}_a(x) - \mathcal{E}_b(x)}_2 \leq 2 \sqrt{m} \left( \sum_{j=1}^k \norm{a_j - b_j}_1 \right)^{1/2} \norm{x}_\infty.
		\end{equation}
	\end{lemma}
	
	\begin{proof}
		First, consider the case $m = 1$.  By the triangle inequality,
		\begin{align}
			\norm{\mathcal{E}_a(x) - \mathcal{E}_b(x)}_1 &\leq \sum_{j=1}^k \norm{a_jxa_j - b_jxb_j}_1 \\
			&\leq \sum_{j=1}^k \left( \norm{a_jx(a_j - b_j)}_1 + \norm{(a_j - b_j)xb_j}_1 \right) \\
			&\leq 2 \norm{x}_\infty \sum_{j=1}^k \norm{a_j - b_j}_1.
		\end{align}
		Since $\mathcal{E}_a$ is subunital and completely positive, $\norm{\mathcal{E}_a(x)}_\infty \leq \norm{x}_\infty$.  Thus, by H{\"o}lder's inequality,
		\begin{align}
			\norm{\mathcal{E}_a(x) - \mathcal{E}_b(x)}_2 &\leq \norm{\mathcal{E}_a(x) - \mathcal{E}_b(x)}_1^{1/2} \norm{\mathcal{E}_a(x) - \mathcal{E}_b(x)}_\infty^{1/2} \\
			&\leq \left( 2 \norm{x}_\infty \sum_{j=1}^k \norm{a_j - b_j}_1 \right)^{1/2} (2 \norm{x}_\infty)^{1/2} \\
			&= 2 \left( \sum_{j=1}^k \norm{a_j - b_j}_1 \right)^{1/2} \norm{x}_\infty.
		\end{align}
		The case for general $m > 1$ follows by summing the squares of the individual coordinates.
	\end{proof}
	
	\begin{corollary}
		Consider the setup of Lemma \ref{lem: building approximate projections}.  Fix $y \in (\mathbb{M}_n)_{\sa}$.  Then for $x \in (\mathbb{M}_n)_{\sa}^m$, we have
		\begin{equation} \label{eq: lower bound for psi by smoothed compression}
			\psi_k^{(n)}(x) \geq \norm*{ \sum_{j=1}^k f_{L,j}(y) x f_{L,j}(y)}_2 - 2 \sqrt{m} \left( 3 \sum_{j=1}^k \left| \frac{1}{k} - \tr_n(f_{L,j}(y)) \right| + \frac{k}{n} \right)^{1/2} \norm{x}_\infty.
		\end{equation}
	\end{corollary}
	
	\begin{proof}
		Let $a_j = f_{L,j}(y)$.  Let $p_j^{\star}$ be given by Lemma \ref{lem: building approximate projections}, and note that $p^{\star} = (p_1^{\star},\dots,p_k^{\star})$ is an element of $\mathcal{P}_k^{(n)}$, and hence
		\begin{equation}
		\psi_k^{(n)}(x) \geq \norm{\mathcal{E}_{p^{\star}}(x)}_2 \geq \norm{\mathcal{E}_a(x)}_2 - \norm{\mathcal{E}_{p^{\star}}(x) - \mathcal{E}_a(x)}_2.
		\end{equation}
		We can then estimate $\norm{\mathcal{E}_{p^{\star}}(x) - \mathcal{E}_a(x)}_2$ using \eqref{eq: compression error} and substitute \eqref{eq: p star error} to control $\sum_{j=1}^k \norm{a_j - p_j^{\star}}_1$.
	\end{proof}
	
	Next, we estimate the smoothed compression associated to $Y_{m+1}(0)$ applied to $Y_{m+1}(t)$.  In particular, this involves analyzing the error in replacing $Y(0) = X(t_0)$ with $X(0) = X$.
	
	\begin{lemma} \label{lem: main example smoothed compression}
		Suppose Assumptions \ref{ass: main example setup}, \ref{ass: setup for Y}.  Then
		\begin{equation} \label{eq: main example smoothed compression}
			\norm*{ \sum_{j=1}^m f_{L,j}(Y_{m+1}(0)) Y(t) f_{L,j}(Y_{m+1}(0)) }_2 \\
			\geq \sqrt{m+1} e^{-t/2} \left( e^{-t_0/2} \left(1 - 2\sqrt{\frac{2k}{\pi L}} \right) - 4 \sqrt{kLt_0^{1/2}} \right).
		\end{equation}
	\end{lemma}
	
	\begin{proof}
		First, recall that $Y(t) = e^{-t/2} Y(0) + (1 - e^{-t})^{1/2} \tilde{S}$ where $\tilde{S}$ is a free semicircular family.  A small computation with free independence shows that
		\begin{equation} \label{eq: orthogonality 1}
			\sum_{j=1}^m f_{L,j}(Y_{m+1}(0)) Y(0) f_{L,j}(Y_{m+1}(0)) \perp \sum_{j=1}^m f_{L,j}(Y_{m+1}(0)) \tilde{S} f_{L,j}(Y_{m+1}(0)).
		\end{equation}
		Therefore,
		\begin{equation} \label{eq: orthogonality 1b}
			\norm*{ \sum_{j=1}^m f_{L,j}(Y_{m+1}(0)) Y(t) f_{L,j}(Y_{m+1}(0)) }_2 \geq e^{-t/2} \norm*{ \sum_{j=1}^m f_{L,j}(Y_{m+1}(0)) Y(0) f_{L,j}(Y_{m+1}(0)) }_2
		\end{equation}
		so it suffices to prove \eqref{eq: main example smoothed compression} in the case that $t = 0$.  Now recall that $Y(0) \sim X(t_0) = e^{-t_0/2} X + (1 - e^{-t_0})^{1/2} S$, so it suffices to show that
		\begin{equation} \label{eq: main example smoothed compression 2}
			\norm*{ \sum_{j=1}^m f_{L,j}(X_{m+1}(t_0)) X(t_0) f_{L,j}(X_{m+1}(t_0)) }_2 \\
			\geq \sqrt{m+1} \left( e^{-t_0/2} \left(1 - 2\sqrt{\frac{2k}{\pi L}} \right) - 4 \sqrt{kLt_0^{1/2}} \right).
		\end{equation}
		Now because $f_{L,j}$ is $L$-Lipschitz, we have by \cite[Lemma 2.1]{Jekel2026FIG} that
		\begin{align}
			\norm{f_{L,j}(X_{m+1}(t_0)) - f_{L,j}(X_{m+1}(0))}_2 &\leq L \norm{X_{m+1}(t_0) - X_{m+1}(0)}_2 \\
			&= L \sqrt{(1 - e^{-t_0}) + (1 - e^{-t_0/2})^2} \\
			&= L \sqrt{2(1 - e^{-t_0/2})} \leq L t_0^{1/2}.
		\end{align}
		Let $a_j = f_{L,j}(X_{m+1}(0))$ and $b_j = f_{L,j}(X_{m+1}(t_0))$.  Thus, by \eqref{eq: compression error} and $\norm{a_j - b_j}_1 \leq \norm{a_j - b_j}_2 \leq Lt_0^{1/2}$, we get
		\begin{align} \label{eq: perturbation error}
			\norm*{ \sum_{j=1}^k b_j X(t_0) b_j - \sum_{j=1}^k a_j X(t_0) a_j }_2 &\leq 2 \sqrt{m+1} \left( \sum_{j=1}^k \norm{b_j - a_j}_1 \right)^{1/2} \norm{X(t_0)}_\infty \\
			&\leq 4 \sqrt{(m+1)kLt_0^{1/2}}.
		\end{align}
		Now since $a_j = f_{L,j}(X_{m+1}(0))$ and $X(0)$ are freely independent of $S$, we apply the same reasoning as \eqref{eq: orthogonality 1}, \eqref{eq: orthogonality 1b} and obtain
		\begin{equation} \label{eq: orthogonality 2}
			\norm*{\sum_{j=1}^k a_j X(t_0) a_j}_2 \geq 
			e^{-t_0/2} \norm*{\sum_{j=1}^k a_j X(0) a_j}_2.
		\end{equation}
		Now $a_j$ commutes with $X_1(0)$, \dots, $X_{m+1}(0)$.  Thus, 
		\begin{equation} \label{eq: commutation}
			\norm*{\sum_{j=1}^k a_j X(0) a_j}_2 = \norm*{X(0) \sum_{j=1}^m a_j^2}_2 \geq \norm{X(0)}_2 - \sqrt{m+1} \norm{X(0)}_\infty \norm*{1 - \sum_{j=1}^k a_j^2 }_2.
		\end{equation}
		Moreover, since $0 \leq \sum_j f_{L,j}^2 \leq 1$, we have $0 \leq \sum_j a_j^2 \leq 1$,
		\begin{equation} \label{eq: function integration deficit 2}
			\norm*{1 - \sum_{j=1}^k a_j^2 }_2^2 \leq \norm*{1 - \sum_{j=1}^k a_j^2 }_1 \leq \int_{-2}^2 \left( 1 - \sum_{j=1}^k f_{L,j}^2 \right)\,d\mu \leq \frac{2k}{\pi L}
		\end{equation}
		using \eqref{eq: function integral deficit} with $\mu$ the semicircular measure and $C = 1/\pi$.  Hence, \eqref{eq: commutation} becomes 
		\begin{equation} \label{eq: commutation 2}
			\norm*{\sum_{j=1}^k a_j X(0) a_j}_2 \geq \sqrt{m+1} \left(1 - 2\sqrt{\frac{2k}{\pi L}} \right).
		\end{equation}
		Combining \eqref{eq: perturbation error}, \eqref{eq: orthogonality 2}, and \eqref{eq: commutation 2} yields \eqref{eq: main example smoothed compression 2}, and hence finishes the proof.
	\end{proof}
	
	\begin{lemma}
		Suppose Assumptions \ref{ass: main example setup}, \ref{ass: setup for Y}, \ref{ass: matrix models for Y}.   Let $k \in \N$.  Then
		\begin{equation} \label{eq: lower bound on expectation of psi}
			\liminf_{n \to \infty} \mathbb{E} \psi_k^{(n)}(Y^{(n)}(t)) \geq \sqrt{m+1} \left( e^{-(t_0+t)/2} - 4e^{-t/2}k^{3/4} t_0^{1/4} - \frac{12}{\sqrt{k}} \right).
		\end{equation}
	\end{lemma}
	
	\begin{proof}
		Let $L = k^2$, and apply \eqref{eq: lower bound for psi by smoothed compression},
		\begin{align}
			\mathbb{E} \psi_k^{(n)}(Y^{(n)}(t)) &\geq \mathbb{E} \norm*{ \sum_{j=1}^m f_{L,j}(Y_{m+1}^{(n)}(0)) Y^{(n)}(t) f_{L,j}(Y_{m+1}^{(n)}(0))}_2 \label{eq: number 1} \\
			& \qquad - 2 \sqrt{m+1} \mathbb{E} \left[ \left( 3 \sum_{j=1}^k \left| \frac{1}{k} - \tr_n(f_{L,j}(Y_{m+1}^{(n)}(0))) \right| + \frac{k}{n} \right)^{1/2} \norm{Y^{(n)}(t)}_\infty \right]. \label{eq: number 2}
		\end{align}
		We first estimate the right-hand side of \eqref{eq: number 1}.  Recall that $Y^{(n)}(0)$ is assumed to be bounded in operator norm by $3$ and $Z^{(n)}$ is a GUE $(m+1)$-tuple which satisfies the tail bound Theorem \ref{thm: GUE operator norm bound}.  Thus, we have that $\limsup_{n \to \infty} \norm{Y^{(n)}(t)}_\infty$ is bounded almost surely.  Moreover, by the asymptotic freeness theorem, $(Y^{(n)}(0),Z^{(n)})$ converges in noncommutative law to $(Y(0),\tilde{S})$ in probability, so in particular $(Y^{(n)}(0),Y^{(n)}(t))$ converges in noncommutative law in probability.  This implies that
		\begin{equation} \label{eq: number 3}
			\lim_{n \to \infty} \norm*{ \sum_{j=1}^m f_{L,j}(Y_{m+1}^{(n)}(0)) Y^{(n)}(t) f_{L,j}(Y_{m+1}^{(n)}(0))}_2 = \norm*{ \sum_{j=1}^m f_{L,j}(Y_{m+1}(0)) Y(t) f_{L,j}(Y_{m+1}(0)) }_2 \text{ in probability}.
		\end{equation}
		Therefore, by Fatou's lemma (which holds for convergence in probability, for instance by passing to almost surely convergent subsequences),
		\begin{equation} \label{eq: number 4}
			\liminf_{n \to \infty} \mathbb{E} \norm*{ \sum_{j=1}^m f_{L,j}(Y_{m+1}^{(n)}(0)) Y^{(n)}(t) f_{L,j}(Y_{m+1}^{(n)}(0))}_2 \geq \norm*{ \sum_{j=1}^m f_{L,j}(Y_{m+1}(0)) Y(t) f_{L,j}(Y_{m+1}(0)) }_2,
		\end{equation}
		and here note that we have a lower bound for the right-hand side in \eqref{eq: main example smoothed compression}, so we have finished estimating the first term \eqref{eq: number 1}.
		
		To estimate \eqref{eq: number 2}, we first apply Cauchy-Schwarz,
		\begin{align}
			&\mathbb{E} \left[ \left( 3 \sum_{j=1}^k \left| \frac{1}{k} - \tr_n(f_{L,j}(Y_{m+1}(0))) \right| + \frac{k}{n} \right)^{1/2} \norm{Y(t)}_\infty \right] \\ \leq& \left( 3 \sum_{j=1}^k \mathbb{E} \left| \frac{1}{k} - \tr_n(f_{L,j}(Y_{m+1}^{(n)}(0))) \right| + \frac{k}{n} \right)^{1/2} \label{eq: number 5} \\
			&\qquad \cdot \mathbb{E}[\norm{Y^{(n)}(t)}_\infty^2]^{1/2}. \label{eq: number 6}
		\end{align}
		Similarly to \eqref{eq: number 3}, convergence in law implies that
		\begin{equation} \label{eq: number 7}
			\lim_{n \to \infty} \left| \frac{1}{k} - \tr_n(f_{L,j}(Y_{m+1}^{(n)}(0))) \right| = \left| \frac{1}{k} - \tau(f_{L,j}(Y_{m+1}(0))) \right| \text{ in probability.}
		\end{equation}
		Since $f_{L,j}$ is bounded, we can apply the bounded convergence theorem (for convergence in probability) to obtain
		\begin{equation} \label{eq: number 8}
			\lim_{n \to \infty} \mathbb{E} \left| \frac{1}{k} - \tr_n(f_{L,j}(Y_{m+1}^{(n)}(0))) \right| = \left| \frac{1}{k} - \tau(f_{L,j}(Y_{m+1}(0))) \right|.
		\end{equation}
		The sum over $j$ of the right-hand side can be bounded similarly to \eqref{eq: function integration deficit 2} using \eqref{eq: function integral deficit}, and hence our final estimate for \eqref{eq: number 5} becomes
		\begin{equation} \label{eq: number 9}
			\lim_{n \to \infty} \left( 3 \sum_{j=1}^k \mathbb{E} \left| \frac{1}{k} - \tr_n(f_{L,j}(Y_{m+1}^{(n)}(0))) \right| \right)^{1/2} \leq \sqrt{\frac{6k}{\pi L}} = \frac{\sqrt{6}}{\sqrt{\pi k}},
		\end{equation}
		where we substituted our choice of $L = k^3$.  To estimate \eqref{eq: number 6}, use the triangle inequality
		\begin{align} \label{eq: number 10}
			\mathbb{E}[\norm{Y^{(n)}(t)}_\infty^2]^{1/2} \leq e^{-t/2} \mathbb{E}[\norm{Y^{(n)}(t)}_\infty^2]^{1/2} + (1 - e^{-t})^{1/2} \mathbb{E}[\norm{Z^{(n)}}_\infty^2]^{1/2}.
		\end{align}
		We assumed that $\norm{Y^{(n)}(t)}_\infty \leq 3$, and due to the tail bounds for GUE matrices, we have $\lim_{n \to \infty} \mathbb{E}[\norm{Z^{(n)}}_\infty^2] = 4$, so
		\begin{equation} \label{eq: number 11}
			\limsup_{n \to \infty} \mathbb{E}[\norm{Y^{(n)}(t)}_\infty^2] \leq [3e^{-t/2} + 2(1 - e^{-t})^{1/2}]^2 \leq [3^2 + 2^2][e^{-t} + (1 - e^{-t})] = 13.
		\end{equation}
		Combining all these results, \eqref{eq: number 1} is estimated by \eqref{eq: number 4} and \eqref{eq: number 2} is estimated by \eqref{eq: number 5} and \eqref{eq: number 6}, which in turn are estimated by \eqref{eq: number 9} and \eqref{eq: number 11}, and therefore
		\begin{align}
			\liminf_{n \to \infty} \mathbb{E} \psi_k^{(n)}(Y^{(n)}(t)) &\geq \norm*{ \sum_{j=1}^m f_{L,j}(Y_{m+1}(0)) Y(t) f_{L,j}(Y_{m+1}(0)) }_2 - 2 \sqrt{m+1} \frac{\sqrt{78}}{\sqrt{\pi k}}.
		\end{align}
		Now we combine this with \eqref{eq: main example smoothed compression}, again evaluating $L = k^2$ to get the lower bound
		\begin{equation}
		\sqrt{m+1} \left( e^{-(t_0+t)/2} \left(1 - \frac{2\sqrt{2}}{\sqrt{\pi} k} \right) - 4e^{-t/2}k^{3/4} t_0^{1/4} - \frac{2\sqrt{78}}{\sqrt{\pi k}} \right).
		\end{equation}
		Finally, estimate
		\begin{equation}
		e^{-(t_0+t)/2} \frac{2\sqrt{2}}{\sqrt{\pi k}} + \frac{2\sqrt{78}}{\sqrt{\pi} k} \leq \frac{2}{k} \frac{\sqrt{78} + \sqrt{2}}{\sqrt{\pi k}} \leq \frac{12}{\sqrt{k}}. \qedhere
		\end{equation}
	\end{proof}
		
	\subsection{Lower bound on Wasserstein distance}
	
	Finally, we deduce a lower bound on the Wasserstein distance of the distributions of $Y^{(n)}(t)$ and $S^{(n)}$.
	
	\begin{proposition}
	With the setup of Assumptions \ref{ass: main example setup}, \ref{ass: setup for Y}, \ref{ass: matrix models for Y} and Notation \ref{not: partition test}, we have
	\begin{equation} \label{eq: final difference of expectations}
	\liminf_{n \to \infty} \left[ \mathbb{E} \psi_k^{(n)}(Y^{(n)}(t)) - \mathbb{E} \psi_k^{(n)}(S^{(n)}) \right] \geq \sqrt{m+1} \left( e^{-(t_0+t)/2} - 4e^{-t/2}k^{3/4} t_0^{1/4} - \frac{13}{\sqrt{k}} \right) - 1 - \sqrt{2 + \log(m+1)}.
	\end{equation}
	\end{proposition}
	
	\begin{proof}
	We apply the upper bound \eqref{eq: final expectation estimate} for $\mathbb{E} \psi_k^{(n)}(S^{(n)})$ with $m + 1$ rather than $m$, and take $n \to \infty$.  We combine this with the lower bound for $\liminf_{n \to \infty} \mathbb{E} \psi_k^{(n)}(Y^{(n)})$ from \eqref{eq: lower bound on expectation of psi}.
	\end{proof}
	
	\begin{proposition} \label{prop: Wasserstein lower bound}
		With the same setup as in the previous proposition, we have the following lower bound on the Wasserstein distance.
		\begin{equation} \label{eq: Wasserstein lower bound}
			\liminf_{n \to \infty} W_2(Y^{(n)}(t),S^{(n)}) \geq \sqrt{m+1} e^{-(t_0+t)/2} - 4(m+1)^{3/4} e^{-t/2} t_0^{1/4} - 14 - \sqrt{2 + \log(m+1)}.
		\end{equation}
	\end{proposition}
	
	\begin{proof}
		Recall that $\psi$ is $2$-Lipschitz.  Let $\tilde{Y} \sim Y^{(n)}(t)$ and $\tilde{S} \sim S^{(n)}$ on the same probability space.  Then
		\begin{equation}
		\mathbb{E} \psi_k^{(n)}(\tilde{Y}) - \mathbb{E} \psi_k^{(n)}(\tilde{S}) \leq \mathbb{E} \norm{\tilde{Y} - \tilde{S}}_2.
		\end{equation}
		Optimizing over the coupling,
		\begin{equation}
		\mathbb{E} \psi_k^{(n)}(Y^{(n)}(t)) - \mathbb{E} \psi_k^{(n)}(S^{(n)}) \leq W_2(Y^{(n)}(t),S^{(n)}).
		\end{equation}
		We then apply the lower bound \eqref{eq: final difference of expectations} for the difference of expectations and substitute $k = m+1$.
	\end{proof}
	
	\section{Conclusion of the main proofs} \label{sec: conclusion}
	
	Finally, we put together our upper estimate for free Fisher information with our results in the classical setting.  Throughout the section, Assumptions \ref{ass: main example setup}, \ref{ass: setup for Y}, \ref{ass: matrix models for Y} are in force.
	
	\begin{proof}[Proof of Theorem \ref{thm: main Fisher}]
	First, we rewrite \eqref{eq: example free Fisher upper bound} for $X(t)$ with $a = e^{-t/2}$ and $b = (1 - e^{-t})^{1/2}$:
	\begin{equation} \label{eq: example free Fisher upper bound 2}
		\Phi_g^*(X(t)) \leq (1 - e^{-t})^{-1} \left[ \frac{2^8 e^{-4t} (m^{1/2} + 1)^4}{(1 - e^{-t})^2 - 2^4 e^{-2t}(m^{1/2} + 1)^2} \right] \text{ when } 4e^{-t}(1 + m^{1/2}) \leq 1 - e^{-t}.
	\end{equation}
	Consequently, applying this to $Y(t) \sim X(t_0 + t)$,
	\begin{equation} \label{eq: example free Fisher upper bound 3}
		\Phi_g^*(Y(t)) \leq (1 - e^{-(t+t_0)})^{-1} \left[ \frac{2^8 e^{-4(t+t_0)} (m^{1/2} + 1)^4}{(1 - e^{-(t+t_0)})^2 - 2^4 e^{-2(t+t_0)}(m^{1/2} + 1)^2} \right].
	\end{equation}
	whenever $4e^{-(t+t_0)}(1 + m^{1/2}) \leq 1 - e^{-(t+t_0)}$ or $5 + 4m^{1/2} \leq e^{t+t_0}$.  In particular, suppose that
	\begin{equation} \label{eq: choice of t}
		t = 3 \log 2 + \log(1 + m^{1/2}) + t', \text{ where } t' \geq 0.
	\end{equation}
	Then using monotonicity and algebraic manipulations
	\begin{equation}
		\Phi_g^*(Y(t)) \leq \frac{8}{7} \left[ \frac{2^{-4} e^{-4t'}}{(1 - 2^{-3} e^{-t'})^2 - 2^{-2} e^{-2t'}} \right] \leq \frac{8}{7} \left[ \frac{2^{-4} e^{-4t'}}{1 - 2^{-2} e^{-t'} + 2^{-6} e^{-2t'} - 2^{-2} e^{-2t'}} \right] \leq \frac{1}{7} e^{-4t'}.
	\end{equation}
	Next, we obtain a lower bound on the limit of classical Fisher information for the matrix models $Y^{(n)}(t)$.  By the information-cost inequality \eqref{eq: normalized info cost},
	\begin{align}
	\liminf_{n \to \infty} \mathcal{I}_g^{(n)}(Y^{(n)}(t))^{1/2} &\geq \liminf_{n \to \infty} W_2(Y^{(n)}(t), S^{(n)}) \\
	&\geq \sqrt{m+1} e^{-(t_0+t)/2} - 4(m+1)^{3/4} e^{-t/2} t_0^{1/4} - 14 - \sqrt{2 + \log(m+1)}. \label{eq: lower Wasserstein bound conclusion}
	\end{align}
	Now we take
	\begin{equation} \label{eq: choice of t0}
		0 < t_0 \leq \frac{1}{(m+1)^3} \leq \frac{1}{8}
	\end{equation}
	Plugging in \eqref{eq: choice of t} and \eqref{eq: choice of t0} and $\sqrt{m+1} \geq (1 + \sqrt{m})/\sqrt{2}$, the right-hand side becomes
	\begin{align}
		\eqref{eq: lower Wasserstein bound conclusion} &\geq (m+1)^{1/4} 2^{-7/4} e^{-t_0/2} e^{-t'/2} - 2^{1/2} (m+1)^{1/2} t_0^{1/4} - 14 - \sqrt{2 + \log(m+1)} \\
		&\geq \frac{1}{4} (m+1)^{1/4} e^{-t'/2} - 16 - \sqrt{2 + \log(m+1)}.
	\end{align}
	Finally, we observe that
	\begin{equation}
		\Phi_g^*(Y(t))^{1/2} \leq \frac{1}{\sqrt{7}} e^{-2t'} < \frac{1}{4} (m+1)^{1/4} e^{-t'/2} - 16 - \sqrt{2 + \log(m+1)},
	\end{equation}
	provided that
	\begin{equation}
		t' \leq \frac{1}{4} \log (m+1), \qquad \frac{1}{4} (m+1)^{1/8} \geq 17 + \sqrt{2 + \log(m+1)}.
	\end{equation}
	Hence, for the asserted range of $t$ in Theorem \ref{thm: main Fisher}, we have $\Phi_g^*(Y(t))^{1/2} < \liminf_{n \to \infty} \mathcal{I}^{(n)}(Y^{(n)}(t))^{1/2}$.  By Observation \ref{obs: subtract converging second moment}, this yields also $\Phi^*(Y(t)) < \liminf_{n \to \infty} \mathcal{I}^{(n)}(Y^{(n)}(t))$.  This yields the asserted separation \eqref{eq: main Fisher}.
	\end{proof}
	
	\begin{proof}[Proof of Theorem \ref{thm: main entropy}]
	Note that $\chi(X(t_0)) > -\infty$ by \cite[Theorem 3.9]{Voiculescu1998}.  The conclusion of Theorem \ref{thm: main Fisher} shows a separation between the classical Fisher information of $Y^{(n)}(t)$ and the free Fisher information of $Y(t)$ on a set of positive measure.  By the same time-substitution as in Observation \ref{obs: entropy conversion}, this means that $Y(0) + t^{1/2}Z$ and $Y^{(n)}(0) + t^{1/2} Z^{(n)}$ exhibit a separation of Fisher information for $t$ in a set of positive measure.  Applying Proposition \ref{prop: classical versus free entropy strict} to $Y(0) = X(t_0)$ completes the proof.
	\end{proof}
	
	\newpage
	
	\bibliographystyle{plain}
	\bibliography{free-entropy}
	
\end{document}